\documentclass[11pt]{article}
\usepackage{amsmath,amsthm,amssymb}
\usepackage[usenames,dvipsnames]{xcolor}
\usepackage{enumerate}
\usepackage{graphicx}
\usepackage{cite}
\usepackage{comment}
\usepackage{oands}
\usepackage{tikz}
\usepackage{changepage}
\usepackage{bbm}
\usepackage{mathtools}
\usepackage[margin=1in]{geometry}
\usepackage[pagewise,mathlines]{lineno}
\usepackage{appendix}
\usepackage{stmaryrd} 
\usepackage{multicol}
\usepackage{microtype}
\usepackage[colorinlistoftodos]{todonotes}
\usepackage{enumitem}

\usepackage[pdftitle={Connectivity properties of the adjacency graph of SLE$_\kappa$ bubbles for $\kappa \in (4,8)$},
  pdfauthor={Ewain Gwynne and Joshua Pfeffer},
colorlinks=true,linkcolor=NavyBlue,urlcolor=RoyalBlue,citecolor=PineGreen,bookmarks=true,bookmarksopen=true,bookmarksopenlevel=2,unicode=true,linktocpage]{hyperref}

\theoremstyle{plain}
\newtheorem{thm}{Theorem}[section]
\newtheorem{cor}[thm]{Corollary}
\newtheorem{lem}[thm]{Lemma}
\newtheorem{prop}[thm]{Proposition}

\def\@rst #1 #2other{#1}
\newcommand\MR[1]{\relax\ifhmode\unskip\spacefactor3000 \space\fi
  \MRhref{\expandafter\@rst #1 other}{#1}}
\newcommand{\MRhref}[2]{\href{http://www.ams.org/mathscinet-getitem?mr=#1}{MR#2}}

\theoremstyle{definition}
\newtheorem{defn}[thm]{Definition}
\newtheorem{example}[thm]{Example}
\newtheorem{remark}[thm]{Remark}
\newtheorem*{claim}{Claim}

\newtheorem{notation}[thm]{Notation}

\numberwithin{equation}{section}

\newcommand{\dsb}{\begin{adjustwidth}{2.5em}{0pt}
\begin{footnotesize}}
\newcommand{\dse}{\end{footnotesize}
\end{adjustwidth}}

\newcommand{\ssb}{\begin{adjustwidth}{2.5em}{0pt}}
\newcommand{\sse}{\end{adjustwidth}}

\newcommand{\aryb}{\begin{eqnarray*}}
\newcommand{\arye}{\end{eqnarray*}}
\def\alb#1\ale{\begin{align*}#1\end{align*}}
\def\allb#1\alle{\begin{align}#1\end{align}}
\newcommand{\eqb}{\begin{equation}}
\newcommand{\eqe}{\end{equation}}
\newcommand{\eqbn}{\begin{equation*}}
\newcommand{\eqen}{\end{equation*}}

\newcommand{\BB}{\mathbbm}
\newcommand{\ol}{\overline}

\newcommand{\op}{\operatorname}

\newcommand{\eqD}{\overset{d}{=}}
\newcommand{\ep}{\epsilon}
\newcommand{\rta}{\rightarrow}

\newcommand{\wt}{\widetilde}
\newcommand{\wh}{\widehat} 
\newcommand{\mcl}{\mathcal}

\newcommand{\bdy}{\partial}

\let\originalleft\left
\let\originalright\right
\renewcommand{\left}{\mathopen{}\mathclose\bgroup\originalleft}
\renewcommand{\right}{\aftergroup\egroup\originalright}

\title{Connectivity properties of the adjacency graph\\ of SLE$_\kappa$ bubbles for $\kappa \in (4,8)$}

\date{  }
\author{
\begin{tabular}{c} Ewain Gwynne\\[-5pt]\small MIT \end{tabular}
\begin{tabular}{c} Joshua Pfeffer\\[-5pt]\small Cambridge \end{tabular} 
}

\begin{document}

\maketitle

\begin{abstract}  
We study the adjacency graph of bubbles---i.e., complementary connected components---of an SLE$_{\kappa}$ curve for $\kappa \in (4,8)$, with two such bubbles considered to be adjacent if their boundaries intersect. We show that this adjacency graph is a.s.\ connected for $\kappa \in  (4,\kappa_0]$, where $\kappa_0 \approx 5.6158$ is defined explicitly. This gives a partial answer to a problem posed by Duplantier, Miller and Sheffield (2014).
Our proof in fact yields a stronger connectivity result for $\kappa  \in (4,\kappa_0]$, which says that there is a Markovian way of finding a path from any fixed bubble to $\infty$. We also show that there is a (non-explicit) $\kappa_1 \in (\kappa_0, 8)$ such that this stronger condition does not hold for $\kappa \in [\kappa_1,8)$.

Our proofs are based on an encoding of SLE$_\kappa$ in terms of a pair of independent $\kappa/4$-stable processes, which allows us to reduce our problem to a problem about stable processes. In fact, due to this encoding, our results can be re-phrased as statements about the connectivity of the adjacency graph of loops when one glues together an independent pair of so-called $\kappa/4$-stable looptrees, as studied, e.g., by Curien and Kortchemski (2014). 

The above encoding comes from the theory of Liouville quantum gravity (LQG), but the paper can be read without any knowledge of LQG if one takes the encoding as a black box. 
\end{abstract}

%AMS subject class: 60J67, 60G52
%Keywords: Schramm-Loewner evolution, stable processes, adjacency graph of bubbles, conneded components, peanosphere, mating of trees

\tableofcontents

\section{Introduction}

\subsection{Overview}

Let $\kappa \in (4,8)$ and let $\eta$ be a chordal Schramm-Loewner evolution (SLE$_\kappa$) curve~\cite{schramm0}, say from 0 to $\infty$ in the upper half-plane $\BB H$.
A \textit{bubble} of $\eta$ is a connected component of $\BB H\setminus \eta$. We declare that two such bubbles are \emph{adjacent} if their boundaries have non-empty intersection. In this paper we will study the adjacency graph of SLE$_\kappa$ bubbles for $\kappa \in (4,8)$. (The analogous graph for $\kappa \in (0,4] \cup [8,\infty)$ is uninteresting since SLE$_\kappa$ has only two complementary connected components for $\kappa \in (0,4]$ and is space-filling for $\kappa \geq 8$~\cite{schramm-sle}).

A natural first question to ask about the adjacency graph of bubbles is whether it is connected, i.e., whether any two bubbles can be joined by a finite path in the graph. 
This question appears as~\cite[Question 11.2]{wedges} and is the SLE analogue of a well-known open problem for Brownian motion, which asks whether the adjacency graph of complementary connected components of a planar Brownian motion (say, stopped at some fixed time) is connected; see, e.g.,~\cite{burdzy-website} or~\cite[Open Problem (4)]{peres-bm}. 

Intuitively, one expects that it is easier for the adjacency graph to be connected when $\kappa$ is closer to 4, since for smaller $\kappa$ the bubbles tend to be larger and the curve itself is ``thinner", e.g., in the sense that it has smaller Hausdorff dimension~\cite{beffara-dim} and a larger set of cut points~\cite{miller-wu-dim}.
 
However, due to the fractal nature of the SLE$_\kappa$ curve, it is not clear a priori whether the adjacency graph should be connected for \emph{any} value of $\kappa \in (4,8)$, even at a heuristic level. 
For instance, the set $S$ of points on the curve which do not lie on the boundary of any bubble has full Hausdorff dimension: indeed, by SLE duality~\cite{zhan-duality1,zhan-duality2,dubedat-duality,ig1,ig4}, the dimension of the boundary of each bubble is equal to the dimension of SLE$_{16/\kappa}$, which is strictly less than the dimension of SLE$_{\kappa}$~\cite{beffara-dim}.
If $S$ contained a non-trivial connected subset, then no path of bubbles in the adjacency graph would be able to cross this subset (c.f.\ Corollary~\ref{cor-dust}). 
One could also worry that there exist pairs of macroscopic bubbles separated by an infinite ``cloud" of small bubbles, so that no finite path of bubbles can join them.
 Figure~\ref{fig-sle-sim} shows a simulation of an SLE curve, which may help the reader to visualize these geometric features.

\begin{figure}
\centering
\includegraphics[width=0.5\textwidth]{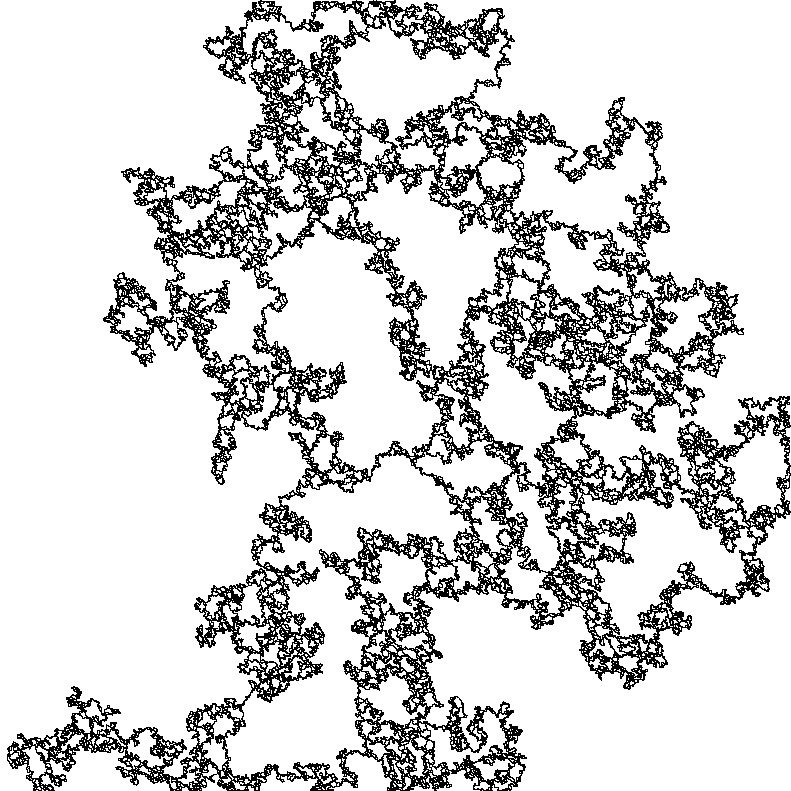}
\caption{An SLE$_6$ in a square domain. Simulation by Jason Miller.} \label{fig-sle-sim}
\end{figure}

In this paper we will give an affimative answer to the above question for an explicit range of values of $\kappa$. With $\psi(x) = \frac{\Gamma'(x)}{\Gamma(x)}$ denoting the digamma function, we have the following. 

\begin{thm}
For each fixed $\kappa \in (4,\kappa_0]$, the adjacency graph of bubbles of a chordal SLE$_\kappa$ curve is almost surely connected, where $\kappa_0 \approx 5.6158$ is the unique solution of the equation $\pi \cot(\pi \kappa/4) + \psi(2-\kappa/4) - \psi(1) = 0$ on the interval $(4,8)$.
\label{main}
\end{thm}
 
We will prove Theorem \ref{main} by proving an stronger condition  (Theorem \ref{path}), which, roughly  speaking, asserts that each bubble of the SLE$_\kappa$ curve is ``connected to infinity'' via an infinite path of bubbles in the adjacency graph which are chosen in a Markovian manner with respect to a natural parametrization of SLE that we introduce in Section~\ref{sec-levy}.  We also show that this stronger condition fails for $\kappa$ sufficiently close to $8$ (Theorem \ref{largekappa}).  
See Section~\ref{sec-open} for some heuristic discussion concerning the values of $\kappa$ for which various connectivity properties hold.

As alluded to earlier, Theorem~\ref{main} tells us that for $\kappa \in (4,\kappa_0]$, there cannot be non-trivial connected subsets of the SLE$_\kappa$ curve which do not intersect the boundary of any bubble. 

\begin{cor} \label{cor-dust}
For $\kappa \in (4,\kappa_0]$, the set of points on a chordal SLE$_\kappa$ curve which do not lie on the boundary of any bubble is almost surely totally disconnected.
\end{cor}
\begin{proof}
Let $\eta$ be a chordal SLE$_{\kappa}$ curve and let $\tau_1$ and $\tau_2$ be forward and reverse stopping times of $\eta$, respectively, with $\tau_1 < \tau_2$ almost surely.  By the reversibility of SLE$_{\kappa}$~\cite{ig3} and the domain Markov property, the conditional law of $\left. \eta \right|_{[\tau_1,\tau_2]}$ conditioned on $\left. \eta \right|_{[0,\tau_1] \cup [\tau_2,\infty)}$ is that of an SLE$_{\kappa}$ curve from $\eta(\tau_1)$ to $\eta(\tau_2)$ in the appropriate connected component $D = D(\tau_1,\tau_2)$ of $\BB H\setminus \eta ([0,\tau_1] \cup [\tau_2,\infty))$.    Theorem \ref{main} applied to this latter SLE curve implies that, almost surely, there does not exist a connected subset of $\eta$ which does not intersect the boundary of any bubble of $\eta$ and which disconnects the interior of $D$, since such a set would disconnect the adjacency graph of bubbles of $\left. \eta \right|_{[\tau_1,\tau_2]}$.  

We can choose a countable collection $\mcl T$ of random pairs of times $(\tau_1,\tau_2)$ such that $\tau_1 < \tau_2$ a.s., $\tau_1$ (resp.\ $\tau_2$) is a forward (resp.\ reverse) stopping time for $\eta$, and the projection of $\mcl T$ onto its first and second coordinates are each dense (e.g., we could conformally map to $\BB D$, parametrize $\eta$ by Minkowski content~\cite{lawler-shef-nat,lawler-zhou-nat,lawler-rezai-nat}, then let $\mcl T$ be the set of pairs of ordered positive rational times). 
If $X$ is a connected subset of $\eta$ with more than one point and we choose $(\tau_1,\tau_2) \in \mcl T$ such that $\tau_1$ (resp.\ $\tau_2$) is sufficiently close to the first (resp.\ last) time that $\eta$ hits $X$, then $X$ will disconnect the interior of the domain $D$ above. Hence the corollary follows from a union bound over all $(\tau_1,\tau_2) \in \mcl T$.  
\end{proof}  

We also mention the recent related work~\cite{aru-sepulveda-2valued}, which studies the \emph{two-valued local sets} of the Gaussian free field---a two-parameter family of random sets constructed from collections of SLE$_4$-type curves. Among other things, the authors determine the parameter values for which the adjacency graph of complementary connected components of these sets are connected, using very different techniques from those of the present paper.

\subsection{Approach and outline}
\label{sec-outline}
 
The key tool in our proof is a pair of independent $\kappa/4$-stable processes $(L,R)$ with only downward jumps, first introduced in~\cite[Corollary 1.19]{wedges}, which encode the geometry of the SLE$_{\kappa}$ curve. 
The existence of these processes reduces our problem to analyzing stable processes rather than SLE$_\kappa$. 
The particular stable processes we consider are characterized by the Laplace transform $\BB E[e^{\lambda L_t}] = \BB E[ e^{\lambda R_t}] = e^{a t \lambda^{\kappa/4}}$, $\forall t,\lambda >0$ or equivalently by the L\'evy measure $b |x|^{-\kappa/4} \BB 1_{(x\leq 0)} \,dx$ for constants $a,b> 0$ which we do not make explicit (see Remark~\ref{remark-scaling}). We refer to~\cite{bertoin-book} for more on stable processes.

We will give the definition of $(L,R)$ in Section~\ref{sec-LRdef}. The definition uses the theory of Liouville quantum gravity (LQG): roughly speaking, $L_t$ (resp.\ $R_t$) for $t\geq 0$ gives the LQG length of the left (resp.\ right) outer boundary of $\eta([0,t])$ minus the LQG length of the interval to the left (resp.\ right) of 0 which is disconnected from $\infty$ by $\eta([0,t])$, when $\eta$ is parametrized by quantum natural time with respect to a certain GFF-type distribution.
The downward jumps of $L$ and $R$ correspond to times at which $\eta$ forms bubbles.
We will review the aspects of LQG theory which are necessary to understand the definition in Section~\ref{sec-lqg}. The reader who is not familiar with LQG can take the existence of $(L,R)$ as a black box throughout the rest of the paper. 

In Section~\ref{sec-markovpath} we use the process $(L,R)$ to formulate a condition for the adjacency graph of SLE$_\kappa$ bubbles which implies connectedness. We will then state Theorems \ref{path} and \ref{largekappa}, which assert that this stronger condition holds for the range of $\kappa$ considered in Theorem \ref{main}, but fails for $\kappa$ sufficiently close to 8.  The remaining sections of the paper will be devoted to proving Theorems \ref{path} and \ref{largekappa}. 

In Section~\ref{sec-reducing}, we explain how to use the Markov and scaling properties of $(L,R)$ to reduce each of Theorems~\ref{path} and~\ref{largekappa} to determining whether the expected logarithm of a certain quantity defined in terms of $(L,R)$ is positive or negative. The remainder of the paper contains the (somewhat tricky) L\'evy process arguments needed to estimate these expectations. Theorem~\ref{path} (which implies Theorem~\ref{main}) is proven in Section~\ref{sec-onebubble} and Theorem~\ref{largekappa} is proven in Section~\ref{sec-converseproof}. In the proofs, we will use several existing results from the L\'evy process literature, including ones from~\cite{chaumont-doney-local-times,chaumont-decomp,bertoin-savov-duality,dk-overshoot,paolella,peskir}. However, since we are interested in certain rather specific times for a pair of independent L\'evy processes, we will also need to prove a number of L\'evy process results by hand. See also Remark~\ref{remark-literature}.

Section~\ref{sec-open} discusses some open problems related to various connectivity properties of the adjacency graph of SLE bubbles.

\subsection{Looptree interpretation}
\label{sec-looptree}
 
Due to the encoding discussed in Section~\ref{sec-outline}, Theorem~\ref{main} can be re-phrased as a statement about the topological space obtained by gluing together a pair of so-called \emph{$\kappa/4$-stable looptrees}, as studied, e.g., in~\cite{curien-kortchemski-looptree-def}. 
We will not directly use looptrees in our proof, so a reader who only wants to see the proof of our results for SLE$_\kappa$ can safely skip this subsection. 

Stable looptrees are obtained from stable L\'evy trees (as defined, e.g., in~\cite{duquesne-legall-levy-trees}) by replacing each branch point (corresponding to the jumps of the L\'evy process) by a circle of perimeter equal to the magnitude of the jump.
In the case of $\kappa/2$-stable processes, this construction is equivalent to the construction of the so-called \emph{forested wedge of weight $\gamma^2-2$} (here $\gamma = 4/\sqrt\kappa$) in~\cite[Figure 1.15, Line 3]{wedges}, except that in the looptree definition the interiors of the disks are not included.
The definition of looptrees/forested wedges is explained in Figure~\ref{fig-looptree}.

\begin{figure}[ht!]
\begin{tabular}{ccc}
\includegraphics[width=0.32\textwidth]{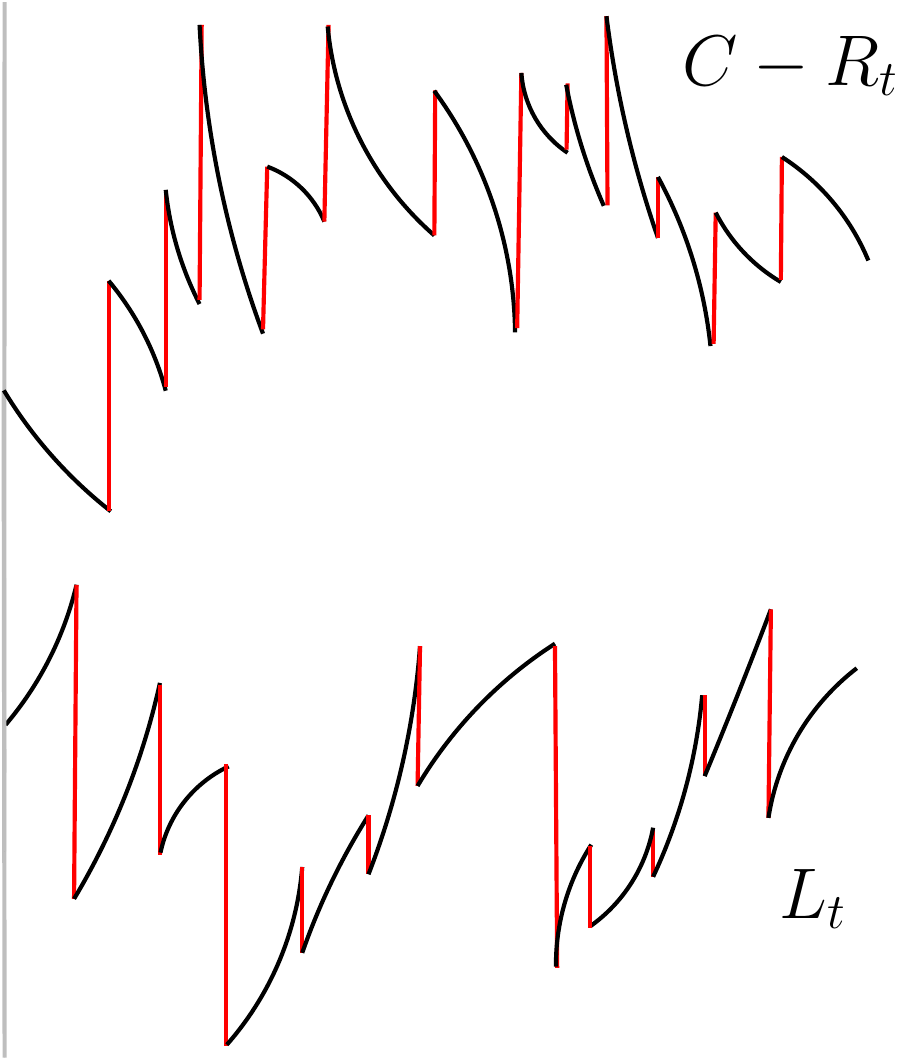}
&
\includegraphics[width=0.32\textwidth]{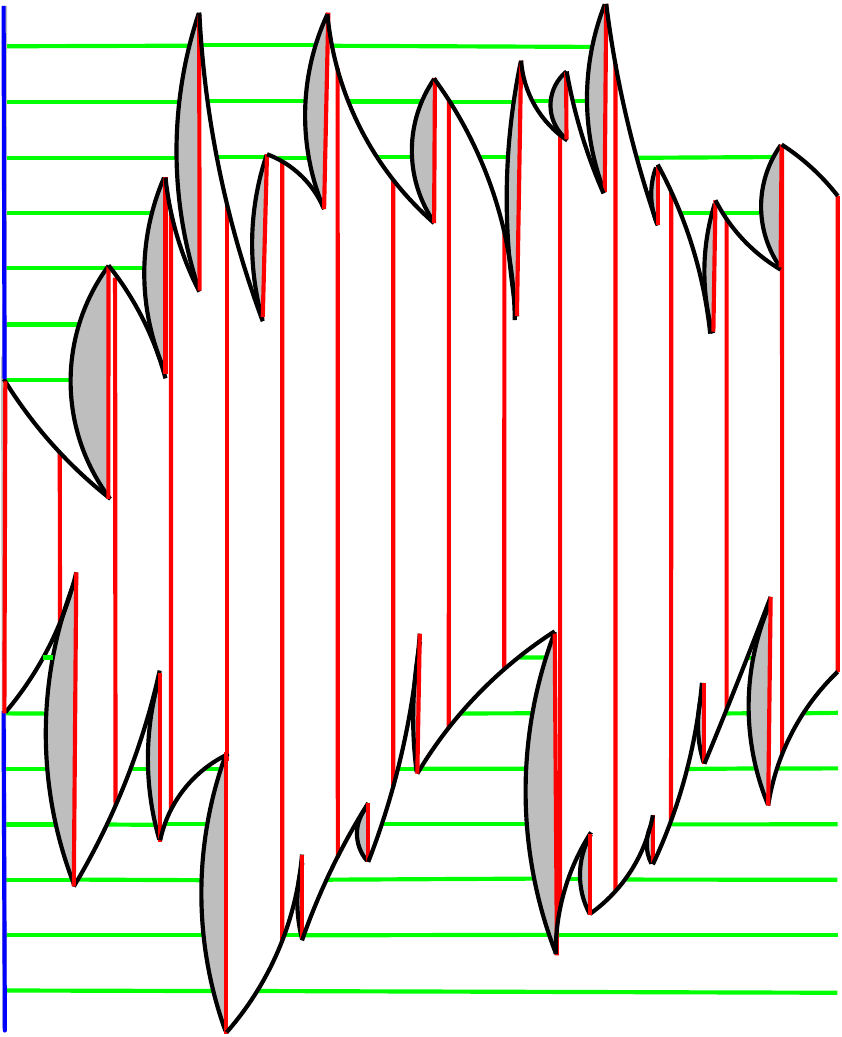}
&
\includegraphics[width=0.36\textwidth]{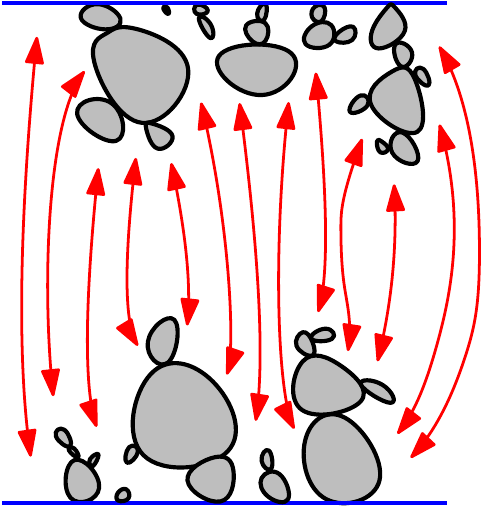}
\end{tabular}
\caption{An illustration of the gluing of two independent $\kappa/4$-stable looptrees described in Corollary \ref{cor-looptree}.  \textbf{Left:} We begin with a pair $(L,R)$ of independent $\kappa/4$-stable processes with only negative jumps.  We can choose a large $C>0$ such that the graphs of $L_t$ and $C-R_t$ do not intersect in some time interval of interest (the particular value of $C$ is unimportant). \textbf{Middle:} For each jump of $L_t$, we draw a black curve underneath the graph of $L$ with the same endpoints as those of the jump, and which intersects each horizontal line only once. The particular geometry of the curves chosen will not affect the topology of the resulting tree. We similarly draw curves corresponding to jumps of $C - R_t$. We then identify pairs of points of the square if they lie on the same horizontal (green) segment that lies below the curve; and similarly for $C-R_t$. This produces a pair of independent forested wedges of weight $\gamma^2-2$. To glue the two forested wedges, we draw vertical (red) segments joining the two graphs, and we connect points on the two graphs that lie on the same vertical segment or on the same jump segment.  \textbf{Right:} The resulting quotient is a pair of forested wedges with outer boundaries identified.  The parts of the forested wedges colored in blue correspond to running minima of $L_t$ and $C - R_t$; or, equivalently, points of $L_t$ and $R_t$ which lie on horizontal green segments that intersect the rays $(-\infty,0)$ and $(C,\infty)$ on the $y$-axis, colored in blue in the middle figure. If we remove the gray interior regions, we obtain a pair of $\kappa/4$-stable looptrees with their outer boundaries identified. We emphasize that the looptrees shown in the right panel are \emph{not} exactly the ones produced from the stable processes in the left and middle panels.  
}
\label{fig-looptree}
\end{figure}

\begin{cor} \label{cor-looptree}
Let $(L,R)$ be a pair of i.i.d.\ $\kappa/4$-stable processes with only downward jumps and let $\mcl G$ be the topological space obtained by gluing the looptrees $\mcl T^L$ and $
\mcl T^R$ associated with $L$ and $R$ together according to the natural length measure along their boundaries which arises from the time parametrizations of $L$ and $R$, as described in Figure~\ref{fig-looptree}.  If $\ell_1$ and $\ell_2$ are two loops, each of which belongs to either $\mcl T^L$ or $\mcl T^R$, we declare that they are adjacent if and only if the corresponding subsets of $\mcl G$ (under the quotient map $\mcl T^L \sqcup \mcl T^R \rta\mcl G$) intersect. If $\kappa \in (4,\kappa_0]$, then the adjacency graph of loops is a.s.\ connected.
\end{cor} 
\begin{proof}
Let $\mcl G^\bullet$ be the topological space obtained by filling in each of the loops of $\mcl G$ with a copy of the unit disk. Equivalently, $\mcl G^\bullet$ can be obtained by replacing each of the loops of $\mcl T^L$ and $\mcl T^R$ with a closed disk, then identifying the resulting trees of disks along their boundaries as we identified $\mcl T^L$ and $\mcl T^R$ to produce $\mcl G$. We note that $\mcl G$ is canonically identified with a closed subset of $\mcl G^\bullet$, namely the image of the boundaries of the trees of disks under the quotient map.
Let $\eta$ be an SLE$_\kappa$ curve. By a slight abuse of notation, we also denote the range of $\eta$ by $\eta$.  It follows from~\cite[Corollary~1.19]{wedges} (see also~\cite[Figure 1.19]{wedges}) that there is a homeomorphism $\BB H\rta \mcl G^\bullet$ which takes $\eta$ to $\mcl G$. 
Here we use the above-mentioned equivalence between looptrees and forested wedges. 
Consequently, $\eta$, viewed as a topological space, is homeomorphic to $\mcl G$ via a homeomorphism under which boundaries of bubbles of $\eta$ correspond to loops of $\mcl T^L$ or $\mcl T^R$. 
The corollary thus follows from Theorem~\ref{main}.
\end{proof}

\subsection*{Acknowledgements} We thank Jean Bertoin, Jason Miller and Scott Sheffield for helpful discussions. We thank two anonymous referees for helpful comments on an earlier version of the paper. 
E.G.\ was partially funded by NSF grant DMS 1209044.  J.P.\ was partially supported
by the National Science Foundation Graduate Research Fellowship under Grant No. 1122374.

\section{A $\kappa/4$-stable process description of SLE$_\kappa$ for $\kappa \in (4,8)$}
\label{sec-levy}

 \subsection{Liouville quantum gravity definitions} 
 \label{sec-lqg}

In order to define the pair of $\kappa/4$-stable processes which encode the geometry of $\eta$, we will need some definitions from the theory of Liouville quantum gravity (LQG). We will not state these definitions precisely here (instead referring to the cited papers), since the only feature of these definitions which is needed in the present paper is Theorem~\ref{LRthm} below.  

Let $\gamma := 4/\sqrt\kappa \in (\sqrt 2 ,2)$. If $D\subset\BB C$ is an open set and $h$ is a random distribution (generalized function) on $D$ which behaves locally like the Gaussian free field on $D$ (see~\cite{shef-gff,ss-contour,ig1,ig4} for more on the GFF) then the \emph{$\gamma$-LQG surface} associated with $h$ is, formally, the random Riemannian surface with Riemann metric tensor $e^{\gamma h(z)} \, (dx^2 + dy^2)$, where $dx^2 + dy^2$ denotes the Euclidean metric tensor. This definition does not make literal sense since $h$ is a distribution, not a pointwise-defined function, so we cannot exponentiate it. However, certain objects associated with $\gamma$-LQG surfaces can be defined rigorously using regularization procedures. 

For example, Duplantier and Sheffield~\cite{shef-kpz} constructed the volume form associated with a $\gamma$-LQG surface, which is a measure $\mu_h$ that can be defined as the limit of regularized versions of $e^{\gamma h(z)} \,dz$ (where $dz$ denotes Lebesgue measure). 
In a similar vein, one can define the \textit{$\gamma$-LQG length measure} $\nu_h$ on certain curves in $\ol D$, including $\bdy D$ and SLE$_{\wh\kappa}$-type curves for $\wh\kappa=\gamma^2$ (or equivalently the outer boundaries of SLE$_\kappa$-type curves, by SLE duality~\cite{zhan-duality1,zhan-duality2,dubedat-duality,ig1,ig4}) which are independent from $h$.
The $\gamma$-LQG length measure can be defined in various ways, e.g., using semi-circle averages of a GFF on a domain with smooth boundary and then confomally mapping to the complement of an SLE$_{\wh\kappa}$ curve~\cite{shef-kpz,shef-zipper} or directly as a Gaussian multiplicative chaos measure with respect to the Minkowski content of the SLE curve~\cite{benoist-lqg-chaos}. 
See also~\cite{rhodes-vargas-review,berestycki-gmt-elementary} for surveys of a more general theory of regularized measures of this form, which dates back to Kahane~\cite{kahane}. 

Also relevant for our purposes is the natural $\gamma$-LQG parametrization of an SLE$_\kappa$ curve $\eta$ sampled independently from $h$; we call this parametrization \emph{quantum natural time}. Parametrizing by quantum natural time is, roughly speaking, the same as parametrizing by ``quantum Minkowski content''. It is the quantum analogue of the so-called natural parametrization of SLE~\cite{lawler-shef-nat,lawler-zhou-nat}. The precise definition of quantum natural time can be found in~\cite[Definition~6.23]{wedges}. 

In this paper, we will always take $D = \BB H$ to be the upper half-plane and $h$ to be the GFF-type distribution corresponding to the so-called \emph{$\frac{4}{\gamma}-\frac{\gamma}{2}$- (equivalently, weight-$\frac{3\gamma^2}{2}-2$) quantum wedge}, which is defined precisely in~\cite[Definition 4.5]{wedges}. Roughly speaking, $h$ is obtained from $\wt h  - \left(\frac{4}{\gamma}-\frac{\gamma}{2} \right) \log |\cdot|$, for $\wt h$ a GFF on $\BB H$ with Neumann boundary conditions, by ``zooming in near the origin" and then re-scaling so that the $\gamma$-LQG mass of $\BB D\cap\BB H$ remains of constant order~\cite[Proposition 4.7(ii)]{wedges}.

 \subsection{Definition of $(L,R)$} 
 \label{sec-LRdef}
 
Let us now suppose that $h$ is the distribution corresponding to a $\frac{4}{\gamma}-\frac{\gamma}{2}$-quantum wedge ($\gamma=4/\sqrt\kappa$), as above, and our SLE$_\kappa$ curve $\eta$ is sampled independently from $h$ and then parametrized by $\gamma$-quantum natural time with respect to $h$. To define the processes $(L,R)$, consider for each $t > 0$ the hull generated by $\eta([0,t])$ (\textit{i.e.}, the closure of the set of points it disconnects from $\infty$) and let $x_t$ and $y_t$ denote the infimum and supremum, respectively, of the set of points where this hull intersects the real line.  We define the \textit{left boundary length} $L_t$ of $\eta$ at time $t$ to be the $\gamma$-LQG length of the boundary arc of the hull from $\eta(t)$ to $x_t$, minus the $\gamma$-LQG length of  the segment $[x_t,0]$.   
Similarly, we define the \textit{right boundary length} $R_t$ of $\eta$ at time $t$ to be the $\gamma$-LQG length of the boundary arc of the hull from $\eta(t)$ to $y_t$, minus the $\gamma$-LQG length of  the segment $[0,y_t]$. See Figure~\ref{fig-bdy-process} for an illustration. One can also thing of $L$ (resp.\ $R$) as measuring the ``net change" of the left (resp.\ right) boundary of the unbounded connected component of $\BB H\setminus \eta([0,t])$ between time 0 and time $t$. The definition of $(L,R)$ is the continuum analogue of the so-called horodistance process for peeling processes on random planar maps, as studied, e.g., in~\cite{curien-glimpse,gwynne-miller-perc}. 

\begin{figure}
\centering
\includegraphics[width=0.7\textwidth]{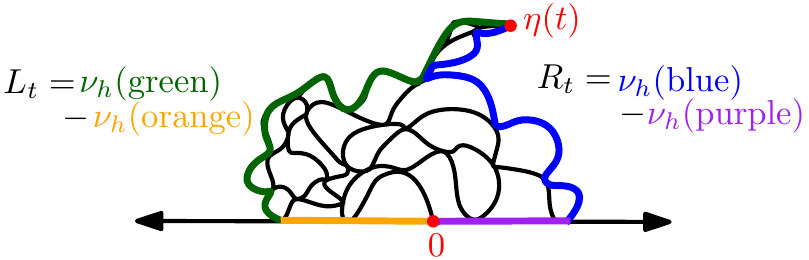}
\caption{The definitions of the processes $L$ and $R$.} \label{fig-bdy-process}
\end{figure}

The following is part of~\cite[Corollary 1.19]{wedges}, and is the only fact from LQG theory which we will need in this paper.

\begin{thm}
The processes $L_t$ and $R_t$ are i.i.d.\ totally asymmetric $\frac{\kappa}{4}$-stable L\'evy processes with only negative jumps.
\label{LRthm}
\end{thm}

\begin{remark} \label{remark-scaling}
Since scaling the time parametrization of a $\kappa/4$-stable L\'evy process gives another $\kappa/4$-stable L\'evy process, 
Theorem~\ref{LRthm} only specifies the law of $(L,R)$ up to a constant re-scaling of time, $(L_t ,R_t) \mapsto (L_{c t} , R_{c t})$ for a constant $c>0$ (or equivalently $(L_t ,R_t) \mapsto c^{4/\kappa} (L_t , R_t)$). 
The properties of $(L,R)$ which we will be interested in do not depend on this scaling, so one can make an arbitrary choice of $c$. 
In Section~\ref{sec-converseproof}, we will fix the scaling in a particularly convenient way. 
\end{remark}

Theorem~\ref{LRthm} is quite powerful because the behavior of these two L\'evy processes neatly encode a lot of the geometry of the SLE$_\kappa$ curve $\eta$; the following set of examples illustrates this connection and will be used repeatedly in the proof of our main results. (The equivalences described in these examples are direct consequences of the theorem.)

\begin{example}
\begin{enumerate}
\item
The time that a bubble of $\eta$ is formed corresponds to a downward jump in either $L_t$ or $R_t$.  For convenience, we call a bubble a \textit{left bubble} or \textit{right bubble} if it corresponds to a downward jump in $L_t$ or $R_t$, respectively.
\item
For $x>0$, let $\rho_x > 0$ be chosen so that the $\gamma$-LQG length of $[0,\rho_x]$ is $x$ (such an $x$ exists since the $\gamma$-LQG length measure has no atoms). The time at which $\eta$ disconnects $\rho_x$ from $\infty$---or, equivalently, the time the bubble with $\rho_x$ on its boundary is formed---is equal to the first time that the process $R_t$ jumps below $-x$. Note that this bubble a.s.\ exists and is unique since $\rho_x$ is independent from $\eta$, so $\eta$ a.s.\ does not hit $\rho_x$. The analogous result holds with $L$ in place of $R$ and with LQG lengths along the negative real axis in place of LQG lengths along the positive real axis.
\item
If $\eta$ forms a left bubble at a time $\tau > 0$, then for $t \in [0,\tau]$ the point $\eta(t)$ lies on the boundary of this bubble if and only if $\inf\{s>t:L_s\leq L_t\} = \tau$, \textit{i.e.}, the time reversed process $L_{\tau  -\cdot} $ attains a running minimum at time $\tau-t$.  The analogous result holds for right bubbles.
\end{enumerate}
\label{examples}
\end{example}

Before introducing one last example describing the geometry of $\eta$ in terms of $(L,R)$, we recall some definitions from the theory of SLE.

\begin{defn}
We say that $t \geq 0$ is a \textit{local cut time} of $\eta$, and $\eta(t)$  a \textit{local cut point}, if $\eta\left( [0,t] \right) \cap \eta\left((t,t+\epsilon]\right) = \emptyset$ for some $\epsilon > 0$.  We call $t$ a \textit{global cut time} and $\eta$ a \textit{global cut point} if $\eta\left( [0,t]\right) \cap \eta\left((t,\infty)\right) = \emptyset$. Since in this paper we will usually want to consider local rather than global cut points, we will refer to local cut points and local cut times simply as \textit{cut points} and \textit{cut times}, respectively.
\end{defn}

\begin{lem} \label{lem-cut-pts}
Almost surely, the set of local cut times for $\eta$ is precisely the set of times $t\geq 0$ for which there exist two connected components (bubbles) $b_1,b_2$ of $\BB H\setminus \eta$ with $\eta(t) \in \bdy b_1\cap \bdy b_2$. Furthermore, if $\bdy b_1\cap \bdy b_2 \not=\emptyset$, then one of $b_1$ or $b_2$ lies to the left of $\eta$ and the other lies to the right of $\eta$. 
\end{lem} 

 See Figure~\ref{pathbubblefigure} below for an illustration of the statement of Lemma~\ref{lem-cut-pts}. 
Lemma~\ref{lem-cut-pts} implies that cut points correspond to edges of the adjacency graph of bubbles. 
The last statement of Lemma~\ref{lem-cut-pts} implies that this adjacency graph is bipartite. 

\begin{proof}[Proof of Lemma~\ref{lem-cut-pts}]
We first argue that a.s.\ every local cut point is an intersection point of the boundaries of two bubbles of $\eta$. 
Choose a countable collection $\mcl T$ (resp.\ $\ol{\mcl T}$) of stopping times for $\eta$ (resp.\ its time reversal) which is a.s.\ dense in $[0,\infty)$. 
By reversibility~\cite{ig4} and the domain Markov property, for any fixed $\tau \in \mcl T$ and $\ol\tau \in \ol{\mcl T}$, on the event $\{\tau < \ol \tau\}$ the conditional law of $\eta|_{[\tau,\ol \tau]}$ given $\eta|_{[0,\tau] \cup [\ol \tau , \infty)}$ is that of an SLE$_\kappa$ from $\eta(\tau)$ to $\eta(\ol \tau)$ in the appropriate connected component of $\BB H\setminus \eta([0,\tau] \cup [\ol \tau , \infty))$.

A time $t > 0$ is a local cut time for $\eta$ if and only if there exists $\tau \in\mcl Q$ and $\ol\tau\in \ol{\mcl Q}$ such that $\tau < t < \ol \tau$ and $t$ is a global cut time for $\eta|_{[\tau , \ol \tau]}$. It therefore suffices to show that a.s.\ every global cut point of $\eta$ is an intersection point of the boundaries of two connected components of $\BB H\setminus \eta$. A global cut point is the same as a point where the left and right outer boundaries of $\eta$ intersect. By~\cite[Theorem 1.4]{ig1}, the left and right outer boundaries $\eta^L$ and $\eta^R$ of $\eta$ can be described as a pair of flow lines of a GFF on $\BB H$. Each of $\eta^L$ and $\eta^R$ is a simple curve, and $\eta^L$ (resp.\ $\eta^R$) does not intersect $(0,\infty)$ (resp.\ $(-\infty,0)$). Consequently, every point of $\eta^L\cap \eta^R$ lies on the boundary of a connected component of $\BB H\setminus \eta^L$ whose boundary intersects $\BB R$ and on the boundary of a connected component of $\BB H\setminus \eta^R$ whose boundary intersects $\BB R$. Each of these connected components is also a connected component of $\BB H\setminus \eta$. 

We remark that the fact that $\eta^L \cap (0,\infty) = \eta^R \cap (-\infty,0) =\emptyset$ shows that $\eta$ a.s.\ does not have any global cut points in $\BB R$, so by the domain Markov property $\eta$ a.s.\ does not have any local cut points $t$ with $\eta(t) \in \eta([0,t))$. By combining this with reversibility, we see that a.s.\ no local cut point of $\eta$ is a double point. 

We now argue that each point on the intersection of two bubbles is a local cut point for $\eta$. 
We first observe that a.s.\ no time at which $\eta$ disconnects a bubble from $\infty$ is a local cut time for $\eta$.
Indeed, each bubble contains a point with rational coordinates and the time at which $\eta$ disconnects such a point from $\infty$ is a stopping time, so a.s.\ is not a local cut time by the domain Markov property.

Now consider two bubbles $b_1,b_2 $ with $\bdy b_1\cap \bdy b_2 \not=\emptyset$, and suppose that $\eta$ finishes tracing $\bdy b_1$ before it finishes tracing $\bdy b_2$. 
Let $\sigma$ be the time at which $\eta$ finishes tracing $\bdy b_1$.  
Let $t \geq 0$ with $\eta(t) \in \bdy b_1\cap \bdy b_2$. By the preceding paragraph, $t\not=\sigma$, so by the definition of $\sigma$, after possibly replacing $t$ with a time $t' < t$ with $\eta(t') = \eta(t)$, we can arrange that $t  <\sigma$. 
Since $\eta$ does not finish tracing $\bdy b_2$  
until after time $\sigma$, $\eta([0,\sigma])$ does not disconnect any point of $b_2$ from $\infty$.
Therefore, for any $\ep \in [0,\sigma-t)$ we can find paths in $\BB H\setminus \eta([0,\sigma-\ep])$ 
from each of the two sides (prime ends) of $\eta(t)$ to $\infty$. This shows that $\eta([0,t))$ and $\eta([t,\sigma-\ep])$ are disjoint. 

Hence $t$ is a local cut time for $\eta$. 

To obtain the second statement of the lemma, we note that our proof that every local time point lies on the boundaries of two distinct bubbles shows that in fact any such cut point lies on the boundaries of two distinct bubbles which lie on opposite sides of $\eta$. The second statement follows from this and the first statement. 
\end{proof}

\begin{example}
In terms of the left and right boundary processes, cut times are times $t$ for which there exists $\ep > 0$ such that $L_s > L_t$ and $R_s > R_t$ for each $s \in (t,t+\ep]$; and global cut times are cut times $t$ such that the processes $L$ and $R$ achieve record minima when they first jump below $L_t$ and $R_t$, respectively, after time $t$.  The processes $L$ and $R$ also identify the two bubbles whose boundaries share a given cut point: if $t$ is the cut time, then the two bubbles are formed at the first times after $t$ that the processes jump below $L_{t}$ and $R_{t}$, respectively.  
Finally, we note that, if $t$ is a global cut time, then the union of the two corresponding bubbles $b,b'$ disconnects the set of bubbles formed before time $t$ from all other bubbles in the adjacency graph. 
\label{fourthexample}. 
\end{example}

\subsection{$(L,R)$-Markovian paths to infinity}
\label{sec-markovpath}

We now use this L\'evy process description of SLE$_\kappa$ for $\kappa \in (4,8)$ to define a ``Markovian path to infinity'' in the adjacency graph of SLE bubbles.

\begin{defn} \label{def-markovpath}
For $\kappa \in (4,8)$,  an \textit{$(L,R)$-Markovian path to infinity} in the adjacency graph of bubbles of $\eta$ is an infinite increasing sequence of stopping times $\tau_1 < \tau_2 < \tau_3 < \cdots $ for $(L,R)$  such that almost surely
\begin{itemize} 
\item $\tau_k \rightarrow \infty$,
\item $\eta$ forms a bubble $b_k$ at each time $\tau_k$ (equivalently, either $L$ or $R$ has a downward jump at time $\tau_k$), and
\item $b_k$ and $b_{k+1}$ are connected in the adjacency graph (i.e., $\bdy b_k\cap\bdy b_{k+1}\not=\emptyset$) for each $k$.
\end{itemize}
\end{defn}

Note that an $(L,R)$-Markovian path to infinity is a \emph{random} path defined for almost every realization of the SLE$_{\kappa}$ curve.  

The existence of $(L,R)$-Markovian paths to infinity is a sufficient condition for connectivity of the adjacency graph of bubbles.

\begin{lem}
Let $\kappa \in (4,8)$, and suppose that, for every stopping time $\zeta$ for $(L,R)$ at which $\eta$ forms a bubble almost surely, the adjacency graph of bubbles admits an $(L,R)$-Markovian path to infinity with $\tau_1 = \zeta$. Then the adjacency graph is connected almost surely.
\label{suffcondition}
\end{lem}

\begin{proof}
The event that the adjacency graph is connected can be expressed as the countable union over all pairs of times $t_1,t_2 \in \BB Q\cap [0,\infty)$ and all $N \in \mathbb{N}$ of the event that $b_{1}$ and $b_{2}$ are joined by a path in the adjacency graph, where for $j\in\{1,2\}$, $b_j$ is the first bubble formed after time $t_j$ that corresponds to a jump of either $L$ or $R$ of magnitude at least $1/N$. Fix such a triple $(t_1,t_2,N)$, and let $\zeta_1$ and $\zeta_2$ be the times at which $\eta$ forms the bubbles $b_1$ and $b_2$, respectively. Since $\eta$ a.s. has arbitrarily large global cut times (see, e.g.\,~\cite[Theorem 1.2]{miller-wu-dim}), we can a.s.\ choose a global cut point $\eta(s)$ with $s>\zeta_1,\zeta_2$.  The point $\eta(s)$ lies on the boundary of two bubbles $b_{3}$ and $b_{4}$ (adjacent to each other) that, as noted in Example \ref{fourthexample} above, together disconnect the set of bubbles formed up to time $s$ from all other bubbles in the adjacency graph. Hence, the $(L,R)$-Markovian paths started at each of $\zeta_1$ and $\zeta_2$ must each pass through one of $b_3$ or $b_4$, which yield finite paths from each of $b_1$ and $b_2$ to either $b_3$ or $b_4$.
\end{proof}

In light of Lemma~\ref{suffcondition}, Theorem~\ref{main} will be an immediate consequence of the following theorem.

\begin{thm}
Suppose $\kappa \in (4,\kappa_0 ]$, with $\kappa_0 \approx  5.6158$ defined as in Theorem \ref{main}. If  $\zeta$ is a stopping time of $(L,R)$ such that $\eta$ forms a bubble at time $\zeta$ almost surely, then the adjacency graph of bubbles admits an $(L,R)$-Markovian path to infinity with $\tau_1 = \zeta$.
\label{path}
\end{thm}

The $(L,R)$-Markovian path appearing in Theorem~\ref{path} is defined explicitly in the proof of Proposition~\ref{reducingtoasinglebubble} below. The times $\tau_k$ can be taken to be stopping times for $\eta$ as well as for $(L,R)$.  

Theorem~\ref{path} gives a strictly stronger connectivity condition for the adjacency graph of bubbles than Theorem~\ref{main}. This stronger condition does not hold for all $\kappa \in (4,8)$. 

\begin{thm}
There exists $\kappa_1 \in (\kappa_0,8)$ such that for $\kappa \in [\kappa_1,8)$, the adjacency graph of bubbles does \emph{not} admit an $(L,R)$-Markovian path to infinity (with any choice of starting time).
\label{largekappa}
\end{thm}

Our proof of Theorem~\ref{largekappa} is based on the fact that a $\kappa/4$-stable process converges in law to Brownian motion as $\kappa$ increases to 8 (Proposition~\ref{convtoBM}), and does not give an explicit formula for $\kappa_1$.

\section{Reducing to an estimate for a single bubble}
\label{sec-reducing}

\begin{figure}
\centering
\includegraphics[width=.7\textwidth]{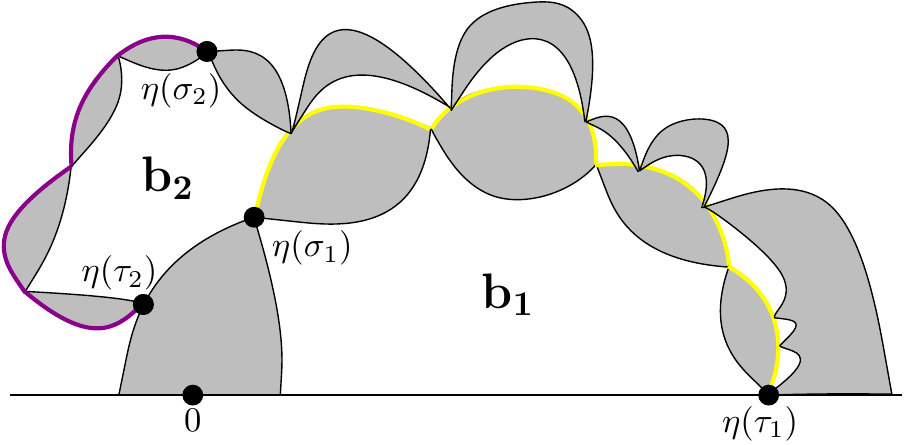}
\caption{The first two bubbles in the path of bubbles defined in the proof of the first half of Proposition \ref{reducingtoasinglebubble}.  The curve $\eta$ on the interval $[0,\tau_2]$ is contained in the regions shaded in gray. The cut point at time $\sigma_1$ corresponds to the edge of the adjacency graph connecting the bubbles $b_1$ and $b_2$. The random variables $X_1$ and $X_2$ defined in~\eqref{eqn-tauinc} give the $\gamma$-LQG lengths of the yellow and purple arcs, respectively.}
\label{pathbubblefigure}
\end{figure}

To prove Theorems \ref{path} and \ref{largekappa}, we first reduce the task of proving the existence or nonexistence of an $(L,R)$-Markovian path to infinity (Definition~\ref{def-markovpath}) to computing an expectation involving a single bubble. 
We first introduce some notation that we will use repeatedly throughout the paper. 

\begin{notation}
For a time $t>0$, we denote by $\sigma(t)$ the smallest $s \in [0,t)$ such that $L_r \geq L_s$ and $R_r \geq R_s$ for all $r \in [s,t)$; or $\sigma(t) = t$ if no such $s$ exists.
\label{sigmanot}
\end{notation} 

We observe that if $\sigma(t) < t$, then $\sigma(t)$ is a cut time for $\eta$ by Example~\ref{fourthexample}, so lies on the boundary of two distinct bubbles formed by $\eta$ by Lemma~\ref{lem-cut-pts}. 

\begin{remark}
Example~\ref{fourthexample} shows that $\sigma(t)$ can equivalently be defined as the smallest $s \in [0,t)$ for which $\eta([0,s)) \cap \eta([s,t)) = \emptyset$ and $\eta([s,t))\cap \BB R = \emptyset$. For a fixed time $t$, the left and right outer boundaries of $\eta([0,t])$ are SLE$_{16/\kappa}$-type curves which a.s.\ intersect each other in every neighborhood of their common starting point: see, e.g.,~\cite{ig4}. Consequently, the description of $\sigma(t)$ in terms of $\eta$ shows that a.s.\ $\sigma(t) < t$. We will not need this fact in our proof, however. One can similarly see from SLE considerations that a.s.\ $\sigma(\tau) < \tau$ if $\tau$ is the first time that $R$ jumps below a specified level, equivalently, the first time that $\eta$ disconnects a certain point of $(0,\infty)$ from $\infty$ (here is is important that we use $[0,t)$ instead of $[0,t]$ in the definition of $\sigma(t)$, since otherwise we would get $\sigma(\tau) = \tau$). As a consequence of Theorem~\ref{onebubble} below, we will obtain a direct proof is this fact which does not use SLE, at least in the case when $\kappa \in (4,\kappa_1]$. 
\end{remark}

\begin{prop}
Let $\kappa \in (4,8)$ and let $\eta$ and $(L,R)$ be as above. Let $\tau$ be the first time that $R$ jumps below $-1$ and let $\sigma = \sigma(\tau)$ (see Notation \ref{sigmanot}).  Equivalently (as noted in Example \ref{examples}), let $\tau$ be the first time that $\eta$ absorbs the point on the positive real axis at $\gamma$-LQG length $1$ from the origin, and let $\sigma$ be the time of the first cut point of $\eta|_{[0,\tau]}$ which lies on the boundary of a bubble of $\eta$ formed after time $\tau$.  If 
\[
\BB{E} \log(L_{\tau} - L_{\sigma}) \geq 0,
\]
then for each stopping time $\zeta$ for $(L,R)$ at which $\eta$ forms a bubble almost surely, there is an $(L,R)$-Markovian path to infinity with $\tau_1 = \zeta$. 

Conversely, let $\mathcal{M}$ denote the set of times in $[0,\tau]$ at which $L$ achieves a record minimum, and suppose that
\begin{equation}
\BB{E} \log\left( \sup_{t \in \mathcal{M}} (L_t - L_{\sigma(t)}) \right) < 0.
\label{converse}
\end{equation}
Then the adjacency graph of bubbles of $\eta$ does not admit an $(L,R)$-Markovian path to infinity. 
\label{reducingtoasinglebubble}
\end{prop}

\begin{remark} \label{remark-sim}
It should be possible to estimate the values of $\kappa$ for which each of the conditions of Proposition~\ref{reducingtoasinglebubble} holds by simulating stable processes numerically. 
However, the times $\sigma(t)$ of Notation~\ref{sigmanot} are \emph{not} continuous functionals of $(L,R)$ with respect to the Skorohod topology. We expect that these times still converge for suitable approximations of $(L,R)$ (see~\cite[Section 1.5]{gms-burger-cone} for related discussion concerning the analogous times for correlated Brownian motions), but the rate of convergence is likely rather slow, which may complicate attempts at simulations. 
\end{remark}

\begin{proof}[Proof of Proposition~\ref{reducingtoasinglebubble}]
First, suppose that $\BB{E} \log(L_{\tau} - L_{\sigma}) > 0$ and suppose we are given a stopping time $\zeta$ for $(L,R)$ at which $\eta$ a.s.\ forms a bubble. We will construct a sequence of stopping times $\zeta = \tau_1 < \tau_2 < \tau_3 < \cdots$ of $(L,R)$ that constitute an $(L,R)$-Markovian path to infinity. We set $\tau_1=\zeta$. We then define the times $\tau_k$ for $k \geq 2$ inductively as follows.  Suppose that we have defined the time $\tau_k$, and that $\eta$ forms a bubble $b_k$ at time $\tau_k$; then we set $\sigma_k = \sigma(\tau_k)$ and
\eqbn
\tau_{k+1} := 
 \begin{cases}
 \inf\{ s > \tau_k : R_s   <   R_{\sigma_k} \} ,\quad &\text{if $b_k$ is a left bubble} \\
 \inf\{ s > \tau_k : L_s  <   L_{\sigma_k} \} ,\quad &\text{if $b_k$ is a right bubble} .
 \end{cases}
\eqen
Equivalently, by Examples~\ref{examples} and~\ref{fourthexample}, $\sigma_k$ is the time of the first cut point of $\eta|_{[0,\tau_k]}$ on the boundary of $b_k$ which also lies on the boundary of a bubble formed after $b_k$, and we choose the next bubble $b_{k+1}$ to be the bubble (other than $b_k$) which has $\eta(\sigma_k)$ on its boundary.   See Figure \ref{pathbubblefigure}.

By definition, $\eta$ forms a bubble at each time $\tau_k$, and the bubbles formed at times $\tau_k$ and $\tau_{k+1}$ are adjacent for each $k$. So, to prove $\tau_1 < \tau_2 < \tau_3 < \cdots$ is an $(L,R)$-Markovian path to infinity, we just need to check that $\tau_k \rightarrow \infty$ almost surely as $k \rightarrow \infty$.  Set
\eqb \label{eqn-tauinc}
X_k := \begin{cases}
 R_{\tau_k} - R_{\sigma_k} &  \text{if $b_k$ is a left bubble} \\
 L_{\tau_k} - L_{\sigma_k} &  \text{if $b_k$ is a right bubble} .
 \end{cases}
\eqe 
If $b_k$ is a right bubble, then by definition $\tau_{k+1}$ is the first time after $\tau_k$ that $L  - L_{\tau_k}$ jumps below $-X_k$.  The same is true if $b_k$ is a left bubble with $L$ replaced by $R$. Hence $X_{k+1}/X_k$ is obtained from the process $X_k^{-1} (L_{\tau_k + \cdot} -L_{\tau_k} , R_{\tau_k + \cdot} - R_{\tau_k})$ in the same manner that $L_\tau-L_\sigma$ is obtained from $(L,R)$, except possibly with the roles of $L$ and $R$ interchanged.  By the strong Markov property, the $\kappa/4$-stable scaling property of $L$ and $R$, and the symmetry between $L$ and $R$, the random variables $X_{k+1}/X_k$ for $k\in\BB N$ are i.i.d., with the same law as $L_\tau-L_\sigma$. 
If $\BB{E} \log(L_{\tau} - L_{\sigma}) > 0$, then the strong law of large numbers implies that a.s.\ $\limsup_{k\rta\infty} \sum_{j=1}^k \log(X_{j+1}/X_j) = \infty$ and therefore $\limsup_{k\rta\infty} X_k = \infty$. 
If $\BB E \log(L_\tau - L_\sigma) = 0$, we again get that a.s.\ $\limsup_{k\rta\infty} \sum_{j=1}^k \log(X_{j+1}/X_j) = \infty$ as follows.
By the Hewitt-Savage zero-one law, the random variable $\limsup_{k\rta\infty} \sum_{j=1}^k \log(X_{j+1}/X_j)$ is a.s.\ equal to a deterministic constant $c\in [-\infty,\infty]$. Since a.s.\ $\limsup_{k\rta\infty} \sum_{j=1}^k \log(X_{j+1}/X_j) = c$, we get that a.s.\ $c - \log(X_2/X_1) = c$. Therefore $c \in \{-\infty, \infty\}$.
By the Chung-Fuchs theorem (see, e.g.,~\cite[Theorem 4.2.7]{durrett}), a.s.\ there are infinitely many $k\in\BB N$ for which $\sum_{j=1}^k \log(X_{j+1}/X_j) > 0$, so we must have $c = \infty$. 
Since $\max_{s \in [0,t]} (|L_s| + |R_s|) < \infty$ for each $t > 0$, this implies that a.s.\ $\tau_k\rta\infty$ as $k\rta\infty$ provided $\BB E \log(L_\tau - L_\sigma) \geq 0$. 

Conversely, suppose that~\eqref{converse} holds. Let $\tau_1 < \tau_2  < \tau_3 < \cdots$ be a sequence of stopping times of $(L,R)$ with $\eta = \tau_1$, such that $\eta$ a.s.\ forms a bubble $b_k$ at each time $\tau_k$, and $b_k$ and $b_{k+1}$ are connected in the adjacency graph for each $k$. 

We claim that $\tau_k$ almost surely does not tend to infinity as $k \rightarrow \infty$.  To prove this claim, we first set  $\sigma_k = \sigma(\tau_k)$
and define $X_k$ as in~\eqref{eqn-tauinc}. For each $k\in\BB N$, $\tau_{k+1}$ is a stopping time greater than $\tau_k$ such that, at time $\tau_{k+1}$, the curve $\eta$ a.s.\ forms a bubble whose boundary shares a cut point with $b_k$.  By Example~\ref{examples}, we can characterize $\tau_{k+1}$ in terms of $(L,R)$ as follows: if $b_k$ is a right bubble, then at time $\tau_{k+1}$, $L_t$ a.s.\ jumps below $-x$ for some random $x\in [L_{\sigma_k}  ,  L_{\tau_k}]$ for the first time after $\tau_k$ (in the special case that $x = L_{\sigma_k}$ almost surely,  the bubble $b_{k+1}$ is the bubble with the cut point $\eta(\sigma_k)$ on its boundary). Equivalently, the process $t\mapsto L_t - L_{\tau_k}$ defined for $t>\tau_k$ achieves a record minimum at $t= \tau_{k+1}$, and $\min_{t\in [0,\tau_{k+1})} (L_t - L_{\tau_k}) \geq -X_k$.
The same is true if $k$ is a left bubble with $L$ replaced by $R$. We deduce from the scaling and Markov properties of $L$ and $R$ that $X_{k+1}/X_k$ is stochastically dominated by  $\sup_{t \in \mathcal{M}} (L_t - L_{\sigma(t)})$. 
Since~\eqref{converse} holds,  the strong law of large numbers implies that a.s.\ $\lim_{k\rta\infty} \sum_{j=1}^k \log(X_{j+1}/X_j) = -\infty$ and therefore that $\lim_{k\rta\infty} X_k = 0$.  

Now, unlike in the first part of the proof, we cannot immediately conclude that  $\tau_k$ almost surely does not tend to infinity as $k \rightarrow \infty$.  The statement $X_k \stackrel{a.s.}{\rightarrow} 0$ says that some measure of boundary length of the bubbles $b_k$ is tending to zero; we want to deduce from this that the path of bubbles must remain in some compact subset of $\BB H$.  

To see this, we observe that Example~\ref{fourthexample} implies that on the event that $\tau_k \rightarrow \infty$, 
it must be the case that for each global cut point $t$ of $\eta$ with $t\geq \tau_1$, 
the sequence of bubbles $\{b_k\}_{k\in\BB N}$ must include one of the bubble with $\eta(t)$ on its boundary. By Lemma \ref{cutptmacro} below,  we can choose a subsequence of bubbles $b_{k_n}$ such that the corresponding random variables $X_{k_n}$ are uniformly bounded from below.  Since $X_{k_n} \rightarrow 0$ almost surely, we deduce that $\tau_k$ almost surely does not tend to infinity, as desired.
\end{proof}

We now state and prove Lemma \ref{cutptmacro}, the missing ingredient we needed to prove Proposition \ref{reducingtoasinglebubble}.

\begin{lem}
Let $\eta$ be an SLE$_{\kappa}$ curve for $\kappa \in (4,8)$. There is a deterministic constant $C>0$ such that a.s.\ there are infinitely many global cut points of $\eta$ such that, if $\tau_l$ and $\tau_r$ are the times $\eta$ forms the left and right bubbles whose boundaries share this cut point, then
\begin{equation}
 \left( R_{\tau_l} - R_{\sigma(\tau_l)} \right) \wedge \left(  L_{\tau_r} - L_{\sigma(\tau_r)} \right) \geq C . 
\label{cutptmacroequation}
\end{equation} 
\label{cutptmacro}
\end{lem}

\begin{proof}
We define times 
\[ r_1 < s_1 < t_1 < r_2 < s_2 < t_2 < r_3 < \cdots \]
inductively as follows. Set $r_0 = s_0 = t_0 = 0$. Inductively, let $r_k$ be the first time $t>t_{k-1}$ such that $R$ attains a running minimum at $t$ and $L_t - \min_{s \in [0,t]} L_s \geq 1$.\footnote{It is not hard to see that such a time always exists: Since the running minimum process of $R$ is a subordinator (Lemma VIII.1 on page 218 of~\cite{bertoin-book}), we can find infinitely many disjoint time intervals that are uniformly large (by the regenerative property of subordinators) and whose endpoints are times at which $R$ attains running minima. The restrictions of $L$  to these time intervals are conditionally independent given $R$, so the 0-1 law implies that, on at least one of these time intervals, the value of $L$ at the right endpoint of the interval will exceed its minimum on that interval by at least one.} 
Let $s_k$ the first global cut time of $\eta$ after time $r_k$; such a cut time exists a.s.\ since $\eta$ a.s.\ has arbitrarily large global cut times (see, e.g.\,~\cite[Theorem 1.2]{miller-wu-dim}). Finally, let 
\[ t_k = \inf\{ t > s_k : L_t < L_{s_k} \} \vee \inf\{ t > s_k : R_t < R_{s_k} \}, \]
\textit{i.e.}, $t_k$ is the larger of the two times at which $\eta$ forms a bubble whose boundary contains the cut point $\eta(t_k)$.  

Using Example~\ref{fourthexample}, each $r_k$ and each $t_k$ is a stopping time for $(L,R)$. By Example~\ref{examples} the random variable of~\eqref{cutptmacroequation} associated to the cut point $s_k$ is a.s.\ determined by $(L,R)|_{[0,t_k]}$. 

We claim that the sequence $\{s_k - r_k\}_{k\in\BB N}$ stochastically dominates an i.i.d.\ sequence of random variables. If we can prove this claim, then the lemma will follow directly from applying Kolmorogorv's 0-1 law.  To show why this claim is true, we first recall how we defined global cut times in terms of $(L,R)$ in Example \ref{fourthexample}. In our setting, since $r_k$ is a stopping time, we can similarly characterize the conditional distribution of $s_k - r_k$ given $L|_{[0,r_k]}$: the law of $s_k - r_k$ is equal to the law of the first global cut time of $(L,R)$ such that the record minimum that $L$ achieves at the first time $\eta$ hits $(-\infty,0]$ after this global cut time is $\leq  L_{r_k} - \min_{s \in [0,r_k]} L_s   $.    Since $L_{r_k} - \min_{s \in [0 ,r_k]} L_s \geq 1$, we deduce by the scaling property of $(L,R)$ that the random variable of~\eqref{cutptmacroequation} associated to the cut point $s_k$ stochastically dominates an a.s.\ positive random variable defined independently of $k$, namely, the random variable~\eqref{cutptmacroequation} associated to the first global cut time of $(L,R)$ such that the record minimum that $L$ achieves after this global cut time is $\leq -1$. This proves our claim, and hence the lemma.
\end{proof}

Proposition \ref{reducingtoasinglebubble} implies that, to prove Theorems \ref{path} and \ref{largekappa}, it is enough to prove the following estimates for a single bubble of an SLE$_{\kappa}$ curve:

\begin{thm}
Fix $\kappa \in (4,\kappa_0]$, where $\kappa_0 \approx  5.6158$ is defined as in Theorem \ref{main}.  Let $\tau$ be the first time that $R$ jumps below $-1$ and $\sigma = \sigma(\tau)$.  Then $\BB{E} \log(L_{\tau} - L_{\sigma}) \geq 0$.
\label{onebubble}
\end{thm}
 
\begin{thm}
There exists $\kappa_1 \in (\kappa_0 , 8)$ such that for $\kappa \in [\kappa_1,8)$, the following is true. 
Let $\mathcal{M}$ denote the set of times $\leq \tau$ at which $L$ achieves a record minimum. Then
\[
\BB{E} \log\left( \sup_{t \in \mathcal{M}} (L_t - L_{\sigma(t)}) \right) < 0.
\]
\label{onebubbleconverse}
\end{thm}

The next section is devoted to proving Theorem \ref{onebubble}; we will prove Theorem \ref{onebubbleconverse} in Section~\ref{sec-converseproof}.

\section{Proof of Theorem \ref{onebubble}}
\label{sec-onebubble}

In this section we prove Theorem \ref{onebubble}. 
In terms of $\eta$, the time $\tau$ in the theorem statement is the first time that $\eta$ absorbs the point $\rho_1$ on the positive real axis at $\gamma$-LQG length $1$ from the origin, and $\sigma$ is the first cut point incident to both the bubble formed at time $\tau$ and some bubble formed at a later time.  In our proof of Theorem \ref{onebubble}, we will also refer to the time $\xi$ at which the process $R$ achieves its minimum on $[0,\tau]$---or, equivalently, the last time $\eta$ hits the positive real axis before time $\tau$.  Figure~\ref{onebubblefigure} illustrates the  definitions of the three times $\xi$, $\sigma$ and $\tau$ in terms of both $\eta$ and the pair of processes $(L,R)$.

\begin{figure}
\begin{tabular}{ccc}
\includegraphics[width=0.36\textwidth]{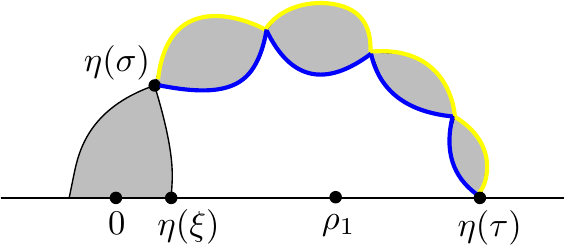}
&
\includegraphics[width=0.32\textwidth]{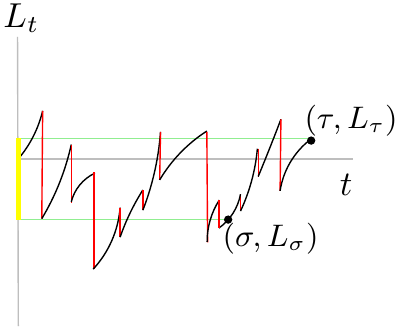}
&
\includegraphics[width=0.32\textwidth]{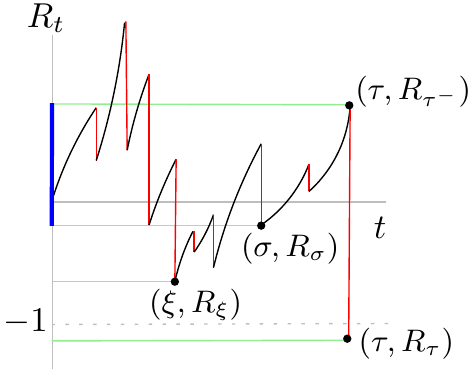}
\end{tabular}
\caption{The times $\xi$, $\sigma$ and $\tau$, defined in terms of $\eta$ (left) and in terms of $(L,R)$ (middle and right). Theorem \ref{onebubble} asserts the the $\gamma$-LQG length of the yellow boundary arc---or, equivalently, the size of the increment in $L$ colored yellow in the middle graph---has nonnegative log expectation.  The first step of our proof of Theorem \ref{onebubble} shows that this quantity stochastically dominates the $\gamma$-LQG length of the blue boundary arc---or, equivalently, the size of the increment in $R$ colored blue in the right graph.}
\label{onebubblefigure}
\end{figure}

Our proof of Theorem \ref{onebubble} consists of three main steps.

\begin{description}
\item[1. Showing that $L_{\tau} - L_{\sigma}$ stochastically dominates $R_{\tau^-} - R_{\sigma}$.]  Since the definition of $\sigma$ is tied closely to that of $\tau$, which depends on $R$ but not on $L$, it is technically easier to study the random variable $R_{\tau^-} - R_{\sigma}$ instead of $L_{\tau} - L_{\sigma}$.  So, we begin by showing that $L_{\tau} - L_{\sigma}$ stochastically dominates $R_{\tau^-} - R_{\sigma}$ (Proposition~\ref{stochasticdomination}), which reduces the task of proving Theorem \ref{onebubble} to showing that $\BB{E} \log(R_{\tau^-} - R_{\sigma}) \geq 0$.  

\item[2. Characterizing the law of $(L,R)$ run backwards from $\tau$ to $\xi$.] Since $\sigma$ is most easily described in terms of the time-reversed processes $L_{(\tau - t)^-}$ and $R_{(\tau - t)^-}$, we next determine the joint law of these time-reversed processes. Proposition~\ref{reversedLevyprocesseslaw} asserts that if we run $L$ and $R$ backward from time $\tau$ until the time $\xi$ at which $R$ reaches its minimum on $[0,\tau)$, then, conditional on $\{R_{\tau^-} - R_{\xi} = r\}$, the law of this pair of time-reversed processes is the same (up to a vertical translation) as that of $(-L,-R)$ run until $-R$ hits the level $-r$.  It follows (Corollary \ref{reversedLevyprocessescorollary}) that the regular conditional distribution of $R_{\tau^-} - R_{\sigma}$ given $\{R_{\tau^-} - R_{\xi} = r\}$ is equal to the law of the value of $R$ at the time $\theta_r$ of the last simultaneous running supremum of $(L,R)$ before $R$ hits the level $r$. By the scaling property of stable processes, this implies that the expectation of $\log(R_{\tau^-} - R_{\sigma})$ is equal to the sums  of the expectations of $\log(R_{\tau^-} - R_{\xi})$ and $\log(R_{\theta_1})$ (equation~\eqref{twoterms} below). 

\item[3. Computing the expectations of $\log(R_{\tau^-} - R_{\xi})$ and $\log(R_{\theta_1})$.]

By the previous step, to prove Theorem \ref{onebubble}, it is enough to show that the sum of the expectations of $\log(R_{\tau^-} - R_{\xi})$ and $\log(R_{\theta_1})$ is positive. 
The first term is easy to handle: we derive the law of $R_{\tau^-} - R_{\xi}$ directly from a result in~\cite{dk-overshoot}. 
To analyze the law of  $\log(R_{\theta_1})$, we use the fact from~\cite{wedges} that the law of $(L,R)$ is equal to a time reparametrization of a pair $(\tilde{L},\tilde{R})$ of correlated Brownian motions to express the law of $R_{\theta_1}$ as that of $\tilde{R}_{\tilde{\theta}_1}$, where $\tilde{\theta}_1$ is the last simultaneous running supremum of $(\tilde{L},\tilde{R})$ before $\tilde{R}$ hits the level $r$. It follows from results in~\cite{evans-cone} and~\cite{hawkes-uniform} that the set of running suprema of a planar Brownian motion has the law of the range of a subordinator whose index we can compute explicitly; hence, we can deduce the law of $R_{\theta_1}$ from the arcsine law for subordinators~\cite{bertoin-sub}.

\end{description}

The next three subsections of the paper are devoted to the proofs of these three main steps. 
 
\begin{remark} \label{remark-literature}
A key difficultly in our proof of Theorem \ref{onebubble} is that, because $\tau$ is a hitting time of $R$ and not $L$, the value $R_\sigma$ is much easier to handle than $L_{\sigma}$.  This is because the time $\sigma$ is more naturally analyzed in terms of $(L,R)$ run backwards from the time $\tau$, and the results in the L\'evy process literature give a nice description of $(L,R)$ run backwards until the running minimum time $\xi$ of $R$ on $[0,\tau]$.  (On this interval, $L$ run backward is just an ordinary L\'evy process, and $R$ run backward is the so-called pre-minimum process of a L\'evy process conditioned to stay positive, whose law is just that of a L\'evy process killed when it reaches a certain random level.) The nature of this result allows us to apply an arcsine law for subordinators to explicitly characterize the law of  $R_{\tau^-} - R_{\sigma}$, but not the law of of $L_{\tau} - L_{\sigma}$, which is the quantity we really care about.  Thus, we need to transfer our analysis of $R_{\tau^-} - R_{\sigma}$ in Steps 2 and 3 to a result for  $L_{\tau} - L_{\sigma}$ by comparing the laws of $L$ and $R$ on $[\sigma,\tau]$ using a crude approximation argument (Lemma~\ref{weight-lemma} below). The existing literature on L\'evy processes is not really helpful here because the time $\sigma$ is neither a stopping time nor a measurable function of a single L\'evy process.
\label{levy-remark}
\end{remark}

\subsection{Showing that $L_{\tau} - L_{\sigma}$ stochastically dominates $R_{\tau^-} - R_{\sigma}$}
\label{sec-dominate}

We now begin with the first step of the proof, which is summarized in the following proposition.

\begin{prop}
The random variable $L_{\tau} - L_{\sigma}$ stochastically dominates $R_{\tau^-} - R_{\sigma}$, \textit{i.e.},
\[\BB{E}\left( g(L_{\tau} - L_{\sigma}) \right) \geq \BB{E}\left( g(R_{\tau^-} - R_{\sigma}) \right)\] for all non-decreasing functions $g$.
\label{stochasticdomination}
\end{prop}

To prove Proposition \ref{stochasticdomination}, we want to characterize the regular conditional distributions of $L$ and $R$ on $[\sigma,\tau]$ given that $\tau - \sigma = t$ and $R_{\sigma}+1 = r$.  Intuitively, we should get (up to vertical translation) a pair of L\'evy processes started at zero and conditioned to stay positive until time $t$, with the second process jumping below $-r$ at time $t$. In the proof that follows, we will precisely define this laws of these two processes,  and show that the law of the second process is equal to the law of the first process weighted by a decreasing function of its value at time $t$ (Lemma \ref{weight-lemma}). By a general probability result (Lemma \ref{stochasticdecreasingfunction}), this property implies that the first process dominates the second, which is exactly the result we want to prove.

Though this heuristic is quite simple, rigorously justifying it requires some technical work; see Remark~\ref{levy-remark} above.
Before delving into the proofs of Lemmas~\ref{weight-lemma} and~\ref{stochasticdecreasingfunction}, which will together imply Proposition \ref{stochasticdomination}, we introduce some definitions and results from the literature that we will use in the proofs of these two lemmas.

First, we will use a discrete approximation of $(L,R)$, so we recall the following consequence of the stable functional central limit theorem.
Let $\{X_j\}_{j \in \BB{N}}$ be an i.i.d.\ sequence of centered random variables with laws supported on $\{1\} \cup \{-m : m \in \BB N\}$ such that
 \eqb
 \BB{P}(X_1=1) = 1-c_0 \quad \text{and} \quad \text{$\BB{P}(X_1 \leq -m) =c_1 m^{-\kappa/4}$ for $m \in \BB{N}$}  , \label{density}
 \eqe
 where the constants $c_0,c_1 > 0$ are chosen so that $\BB{E} X_1 = 0$, and let $S_n = \sum_{i=1}^n X_i$ be the associated heavy-tailed random walk. Then, for some constant $C > 0$ (recall Remark~\ref{remark-scaling}), the rescaled walk
\begin{equation}
W^{(n)}_t := C n^{-4/\kappa}  S_{\lfloor nt \rfloor} 
\label{rescaledwalk}
\end{equation}
converges in distribution to $L$ in the space of c\`adl\`ag functions $\mathcal{D}([0,\infty),\BB{R})$ with respect to the Skorohod topology (see, e.g.,~\cite{js-limit-thm}). 

Second, to analyze stochastic processes restricted to bounded intervals as random variables with values in $\mathcal{D}([0,\infty),\BB{R})$, we introduce the following convention: if $X: [0,\infty) \rightarrow \BB{R}$ is a c\`adl\`ag stochastic process and $a<b$ are positive real numbers, then we define the process $X$ on the interval $[a,b)$ as the process $Y: [0,\infty) \rightarrow \BB{R}$ with $Y_t = X_{t+a}$ for $t \in [0,b-a)$ and $Y_t = 0$ for $t \geq b-a$. Similarly, we define the process $X$ on the interval $[a,b]$ as the process $Y: [0,\infty) \rightarrow \BB{R}$ with $Y_t = X_{t+a}$ for $t \in [0,b-a]$ and $Y_t = X_b$ for $t \geq b-a$.

Third, our proof of Lemma \ref{weight-lemma} below uses two approximation procedures: the discrete approximations of L\'evy processes by random walks given by~\eqref{rescaledwalk}, and an approximation of the condition that the processes stay positive by a condition that they stay above $-\epsilon$. To take the necessary limits of the associated regular condition distributions, we will repeatedly use the following elementary lemma.

\begin{lem} \label{lem-cond-law-conv}
Let $(X_n,Y_n)$ be a sequence of pairs of random variables taking values in a product of separable metric spaces $\Omega_X\times\Omega_Y$ and let $(X,Y)$ be another such pair of random variables such that $(X_n,Y_n) \rta (X,Y)$ in law. Suppose further that there is a family of probability measures $\{P_y : y\in \Omega_Y\}$ on $\Omega_X$, indexed by $\Omega_Y$, such that for each bounded continuous function $f : \Omega_X\rta\BB R $,  
\eqb \label{eqn-cond-law-hyp}
\left( \BB E\left[f(X_n) \, |\, Y_n  \right] , Y_n \right)  \rta \left( \BB E_{P_Y}(f) , Y \right)  \quad \text{in law}.
\eqe
Then $P_Y$ is the regular conditional law of $X$ given $Y$.
\end{lem}
\begin{proof}
Let $g : \Omega_Y\rta \BB R$ be a bounded continuous function. Then for each bounded continuous function $f : \Omega_X\rta\BB R $,
\alb
\BB E\left[f(X) g(Y) \right]
&= \lim_{n\rta \infty} \BB E\left[f(X_n) g(Y_n) \right] \quad \text{(since $(X_n,Y_n) \rta (X,Y)$ in law)} \notag\\
&= \lim_{n\rta\infty} \BB E\left[ \BB E\left[f(X_n) \, |\, Y_n \right]  g(Y_n)   \right] \notag \\
&=  \BB E\left[ \BB E_{P_Y}(f) g(Y)\right] \quad \text{(by~\eqref{eqn-cond-law-hyp})} .
\ale
%By linearity, we have $\BB E\left(F(X,Y) \right) = \BB E\left( \BB E_{P_Y}(F(\cdot,Y) ) \right)$ for every function $F: \Omega_X\times\Omega_Y$ which is a finite linear combination of functions of the form $f(\cdot) g(\cdot)$ for $f :\Omega_X\rta \BB R$ and $Y : \Omega_Y\rta\BB R$ bounded and continuous.
By the functional monotone class theorem, this implies that $\BB E\left[F(X,Y) \right] = \BB E\left[\BB E_{P_Y}(F(\cdot,Y) ) \right]$ for every bounded Borel-measurable function $F$ on $\Omega_X\times\Omega_Y$. Thus the statement of the lemma holds.  
\end{proof}

Lemma~\ref{lem-cond-law-conv} and its proof are essentially identical to those of~\cite[Lemma 5.10]{gms-burger-cone}, except that the statement of~\cite[Lemma 5.10]{gms-burger-cone} is not quite correct since it only requires $\BB E\left[f(X_n) \, |\, Y_n  \right] \rta \BB E_{P_Y}(f)$ in law instead of~\eqref{eqn-cond-law-hyp} (all of the uses of the lemma in~\cite{gms-burger-cone}, however, are in situations where~\eqref{eqn-cond-law-hyp} is satisfied). We thank an anonymous referee for pointing out this error.

Finally, in order to take the $\epsilon \rightarrow 0$ limit of the processes conditioned to stay above $-\ep$, we will need to know that the law of a L\'evy process on $[0,t)$ started at $\epsilon$ and conditioned to stay positive on $[0,t)$ converges to a limit (in the Skorohod topology) as $\epsilon \rta 0$.  This is the content of the following lemma, which appears as Lemma 4 in~\cite{chaumont-doney-local-times}:

\begin{lem}
The law of a L\'evy process on $[0,t)$ started at $\epsilon$ and conditioned to stay positive on $[0,t)$ converges to a limit $L^{+}_{\cdot | t}$ (in the Skorohod topology) as $\epsilon \rta 0$; we call this limiting process the \textit{meander} with length $t$.
\label{meander}
\end{lem}

We can now characterize precisely  regular conditional distributions of $L$ and $R$ on $[\sigma,\tau]$ given that $\tau - \sigma = t$ and $R_{\sigma}+1 = r$.  

\begin{lem}
The regular conditional distributions of $L_{\sigma + \cdot} - L_{\sigma}$ and $R_{\sigma + \cdot} - R_{\sigma}$ on $[0,\tau-\sigma)$ given $\{\tau - \sigma = t\} \cap \{R_{\sigma}+1 = r\}$ are given, respectively, by the law of the meander  $L^{+}_{\cdot | t} $ and the law of the meander $L^{+ }_{\cdot | t} $ weighted by
\eqb 
\frac{\left(L^{+}_{t^- | t}  + r \right)^{-\kappa/4}}{\BB{E}\left( \left(L^{+}_{t^- | t}  + r \right)^{-\kappa/4}\right)}.\label{weightlimitlimit} 
\eqe 
\label{weight-lemma}
\end{lem}

\begin{proof}
Let $L^{(n)}$ and $R^{(n)}$ be independent copies of the rescaled walk $W^{(n)}$ of~\eqref{rescaledwalk}. Also, for fixed $r,\epsilon>0$, let $L^{(n,r,\epsilon)}$ and $R^{(n,r,\epsilon)}$ be obtained from the independent processes $L^{(n)} + \ep$ and $R^{(n)} + \ep$ by conditioning both processes to stay positive until the first time $\tau^{(n,r,\epsilon)}$ that the process $R^{(n,r,\epsilon)}$ hits the level $-r$. We define the processes $L^{(r,\epsilon)}$ and $R^{(r,\epsilon)}$ and the stopping time $\tau^{(r,\epsilon)}$ analogously with $(L,R)$ in place of $(L^{(n)}, R^{(n)})$. Since we are conditioning on a positive probability event, \eqb (L^{(n,r,\ep)} , R^{(n,r,\ep)} , \tau^{(n,r,\ep)}) \stackrel{\mathcal{L}}{\longrightarrow} (L^{(r,\ep)} , R^{(r,\ep)} , \tau^{(r,\ep)}) \label{nrepsilon} \eqe

By the choice of step distribution in~\eqref{density} and Bayes' rule, 
\begin{enumerate}[label=\textbf{(\Roman*)}]
\item
the regular conditional distribution of $L^{(n,r,\epsilon)}$ on the interval $\left[0,\tau^{(n,r,\epsilon)}-1/n\right)$ given $\{\tau^{(n,r,\epsilon)} =t\}$, weighted by
\eqb
\left.\left( L^{(n,r,\epsilon)}_{t - 1/n} + r \right)^{-\kappa/4} \middle/ \BB{E}\left( \left( L^{(n,r,\epsilon)}_{t - 1/n} + r \right)^{-\kappa/4}\right)\right. .
\label{weight-nre}
\eqe
\label{i}
\end{enumerate}
equals, for a.e.\ $t$ (a.e.\ taken w.r.t.\ the law of $\tau^{(n,r,\ep)}$), 
\begin{enumerate}[resume*]
\item
the regular conditional distribution of $R^{(n,r,\epsilon)}$ on the interval $\left[0,\tau^{(n,r,\epsilon)}-1/n\right)$ given $\{\tau^{(n,r,\ep)} = t\}$.
\label{ii}
\end{enumerate}
To prove the lemma, we would like to use this equality in distribution and take the limit as $n \rightarrow \infty$ and $\epsilon \rightarrow 0$.  The $n \rightarrow \infty$ limit is fairly straightforward.  Consider the family  $\{\mu_t : t\in \BB{R} \} $ of probability measures on  $\mathcal{D}([0,\infty),\BB{R})$ with $\mu_t$ defined as the distribution of a L\'evy process started at $\epsilon$ and conditioned to stay positive until time  $t$.  It is easy to see that the  joint law of $(L^{(n,r,\ep)} , R^{(n,r,\ep)} , \tau^{(n,r,\ep)})$ and the conditional law of $L^{(n,r,\epsilon)}$ given $\tau^{(n,r,\epsilon)}$ tends to $(L^{(r,\ep)}, R^{(r,\ep)},\tau^{(r,\ep)}, \mu_{\tau^{(r,\epsilon)}})$.  Thus, the joint law of $\tau^{(n,r,\epsilon)}$ and the conditional law of $L^{(n,r,\epsilon)}$  given $\tau^{(n,r,\epsilon)}$ weighted by \eqref{weight-nre} tends to $\mu_{\tau^{(r,\epsilon)}}$ weighted by
\eqb
\left. \left( L^{(r,\epsilon)}_{{\tau^{(r,\epsilon)}}^-} + r \right)^{-\kappa/4} \middle/ \BB{E}\left( \left( L^{(r,\epsilon)}_{{\tau^{(r,\epsilon)}}^-} + r \right)^{-\kappa/4}\right)\right. 
\eqe
Hence, by \eqref{nrepsilon} and Lemma~\ref{lem-cond-law-conv},  \ref{i} converges  to
\begin{enumerate}[resume*]
\item
the regular conditional distribution of $L^{(r,\epsilon)}$ on the interval $\left[0,\tau^{(r,\epsilon)}\right)$ given $\{\tau^{(r,\epsilon)} = t\}$, weighted by
\eqb
\left. \left( L^{(r,\epsilon)}_{t^-} + r \right)^{-\kappa/4} \middle/ \BB{E}\left( \left( L^{(r,\epsilon)}_{t^-} + r \right)^{-\kappa/4}\right)\right. 
\label{weightrn}
\eqe
\label{iii}
\end{enumerate}
This implies that \ref{iii} also equals, for a.e. $t$, the weak limit of \ref{ii} as $n \rightarrow \infty$. Hence,
\begin{enumerate}[resume*]
\item
the regular conditional distribution of $R^{(r,\epsilon)}$ on the interval $\left[0,\tau^{(r,\epsilon)}\right)$ given $\{\tau^{(r,\epsilon)} = t\}$ 
\label{iv}
\end{enumerate}
exists and is equal in law to \ref{iii}.  

Next, we would like to take $\epsilon \rightarrow 0$.  By Lemma \ref{meander}, the regular conditional distribution of $L^{(r,\epsilon)}$ on $\left.\left[0,\tau^{(r,\epsilon)}\right)\right.$ given $\{\tau^{(r,\epsilon)} = t\}$ given by $\mu_t$ converges weakly as $\epsilon \rightarrow 0$ to the meander $L^{+}_{\cdot | t}$ with length $t$. By the equality of the laws \ref{iii} and \ref{iv}, Lemma~\ref{meander} also implies that \ref{iv} converges weakly as $\epsilon \rightarrow 0$.  Taking $\epsilon \rightarrow 0$ in \ref{iii} and \ref{iv}, we deduce that
\begin{enumerate}[resume*]
\item
the law  of $L^{+}_{\cdot | t} $, weighted by~\eqref{weightlimitlimit}
\label{v}
\end{enumerate}
is equal to 
\begin{enumerate}[resume*]
\item the weak limit of \ref{iv} as $\epsilon \rightarrow 0$.
\label{vi}
\end{enumerate}
So, to prove the lemma, it is enough to prove the following claim:

\begin{claim}
The regular conditional distributions of $L_{\sigma + \cdot} - L_{\sigma}$ and $R_{\sigma + \cdot} - R_{\sigma}$ on $[\sigma,\tau)$ given $\{\tau - \sigma = t\} \cap \{R_{\sigma}+1 = r\}$ are given, respectively, by  the law of $L^{+}_{\cdot | t}$ and \ref{vi} with $r = 1 + R_{\sigma}$. 
\end{claim}

Fix $s , \delta > 0$.  For $ (\mathcal L , \mathcal R) \in \mathcal{D}([0,s+\delta],\BB{R}^2)$, the regular conditional distribution of $(L_{\sigma + \delta + \cdot}- L_{\sigma + \delta},R_{\sigma + \delta + \cdot} - R_{\sigma + \delta})$ given that $ \sigma = s $ and $\left. (L,R) \right|_{[0,\sigma+\delta]} = (\mcl L ,\mcl R)$  (when these conditions are compatible) is the law of a pair of independent L\'evy processes conditioned to stay above $\mathcal L_{s}- \mathcal L_{s+\delta}$ and $\mathcal R_{s}- \mathcal R_{s+\delta}$, respectively, until the first time the second process jumps below $-1-\mathcal R_{s+\delta}$.  Hence, considering the processes $L$ and $R$ separately, we have the following. 
\begin{itemize}
\item
The regular conditional distribution of $L_{\sigma + \delta + \cdot}- L_{\sigma + \delta}$ given $\{\sigma = s\}$,  $\{\tau = w\}$, and $\{ \left. (L,R) \right|_{[0,\sigma + \delta]} = (\mcl L,\mcl R)\}$ (when these conditions are compatible) is that of a L\'evy process started from 0 and conditioned to stay above $\mcl L_{s}- \mcl L_{s + \delta}$ until time $w-s-\delta$. By L\'evy scaling, scaling the time parameter of this process by $\frac{w-s}{w-s-\delta}$ and space by $\left(\frac{w-s}{w-s-\delta}\right)^{4/\kappa}$ yields the law of a L\'evy process conditioned to stay above $(\mcl L_{s}- \mcl L_{s + \delta}) \left(\frac{w-s}{w-s-\delta}\right)^{4/\kappa}$ until time $w-s$. 
By Lemma~\ref{meander}, this regular conditional law converges a.s.\ as $\delta \rta 0$ (weakly, w.r.t.\ the Skorokhod topology) to the law of a L\'evy meander $L^{+}_{\cdot | w-s}$ with length $w-s$. 
Obviously, $(L,R)|_{[0,\sigma+\delta]} \rta (L,R)|_{[0,\sigma]}$ and $L_{\sigma+\delta+\cdot} - L_{\sigma+\delta} \rta L_{\sigma+\cdot}  - L_\sigma$ a.s.\ w.r.t.\ the Skorokhod topology. 
By sending $\delta \rightarrow 0$ and applying Lemma~\ref{lem-cond-law-conv}, we deduce that the regular conditional distribution of $L_{\sigma + \cdot}- L_{\sigma}$ given $\{\sigma = s\}$,  $\{\tau = w\}$, and $\{ \left. (L,R) \right|_{[0,\sigma]} = (\mcl L,\mcl R)\}$ is the law of the L\'evy meander $L^{+}_{\cdot | w-s}$.

\item
The regular conditional distribution of $R_{\sigma + \delta + \cdot}- R_{\sigma + \delta}$ given $\{\sigma = s\}$, $\{\tau = w\}$, and $\{ \left. (L,R) \right|_{[0,\sigma + \delta]} = (\mcl L,\mcl R)\}$ (when these conditions are compatible) is that of a L\'evy process conditioned to stay above $\mcl R_{s}- \mcl R_{s + \delta}$ until jumping below $-1 - \mcl R_{s + \delta}$ at time $w-s - \delta$. By L\,evy scaling, scaling the time parameter of this process by $\frac{w-s}{w-s-\delta}$ and space by $\left(\frac{w-s}{w-s-\delta}\right)^{4/\kappa}$ yields the law of a L\'evy process conditioned to stay above 
 $(\mcl R_{s}- \mcl R_{s + \delta}) \left(\frac{w-s}{w-s-\delta}\right)^{4/\kappa}$ until jumping below $(-1 - \mcl R_{s + \delta}) \left(\frac{w-s}{w-s-\delta}\right)^{4/\kappa}$ at time $w-s$.  
 
Vertically translating by $\mcl R_{s+\delta}-\mcl R_s$ yields exactly \ref{iv} with $\epsilon$, $r$ and $t$ given by $(\mcl R_{s + \delta} - \mcl R_s)  \left(\frac{w-s}{w-s-\delta}\right)^{4/\kappa}$, $(1 + \mcl R_s)  \left(\frac{w-s}{w-s-\delta}\right)^{4/\kappa}$, and $w - s$, respectively.

Taking $\delta \rightarrow 0$ and applying Lemmas~\ref{lem-cond-law-conv} and \ref{meander}, we deduce that the regular conditional distribution of $R_{\sigma + \cdot}- R_{\sigma}$ on $[0,w-s)$ given $\left. (L,R) \right|_{[0,s]}$, $\{\sigma = s\}$, and $\{\tau = w\}$ is given by \ref{vi} with $r$ and $t$ replaced by $1 + R_{\sigma}$ and $w-s$, respectively.
\end{itemize}
This proves the claim, and hence the lemma. 
\end{proof}

The result of Proposition~\ref{stochasticdomination} is now a simple application of the following elementary probability fact, originally due to Harris~\cite{harris-inequality}. 

\begin{lem}[\cite{harris-inequality}]
 Let $X$ be a real-valued random variable, let $f: \BB{R} \rightarrow \BB{R}$ be a non-increasing function with $\BB{E} f(X) = 1$, and let $g: \BB{R} \rightarrow \BB{R}$ be a non-decreasing function. Then 
\eqb
\BB{E}(f(X) g(X)) \leq \BB{E} g(X).
\label{harris}
\eqe
\label{stochasticdecreasingfunction}
\end{lem}

To deduce Proposition~\ref{stochasticdomination} from Lemma~\ref{stochasticdecreasingfunction}, we observe that Lemma~\ref{weight-lemma} implies that for non-decreasing $g$, the expectations of $g(R_{\tau^-} - R_{\sigma} ) $ and $g(L_{\tau^-} - L_{\sigma} ) $ with respect to the regular conditional probability given $\{\tau - \sigma = t\} \cap \{R_{\sigma}+1 = r\}$ are equal to the left and right-hand sides of \eqref{harris}, respectively, with $X = L^{+}_{\cdot |t}$ and $f(x) =C (x+ r)^{-\kappa/4}$ for $C = \BB{E}\left( \left(L^{+}_{t^- | t}  + r \right)^{-\kappa/4}\right)$.

\begin{comment}

\begin{proof}
We write
\eqb \label{eqn-decreasingdiff}
\BB{E}(f(X) g(X)) - \BB{E}(g(X)) = \BB{E} \left( (f(X)-1) \mathbf{1}_{f(X) > 1} g(X) \right) - \BB{E} \left( (1-f(X)) \mathbf{1}_{f(X) \leq 1} g(X) \right)  
\eqe
Observe that 
\eqb \label{eqn-decreasingmean}
\BB{E} \left( (f(X)-1) \mathbf{1}_{f(X) > 1} \right) = \BB{E} \left( (1-f(X)) \mathbf{1}_{f(X) \leq 1} \right)
\eqe
since the difference of the two expectations is $E(f(X) -1)$, which equals zero by hypothesis. On the other hand, since $f$ is non-increasing and $g$ is non-decreasing, $g(x) \leq g(y)$ for any $x,y$ with $f(x)>1$ and $f(y) \leq 1$. Combining this with~\eqref{eqn-decreasingmean} shows that
\[
\BB{E} \left( (f(X)-1) \mathbf{1}_{f(X) > 1} g(X) \right) \leq \BB{E} \left( (1-f(X)) \mathbf{1}_{f(X) \leq 1} g(X) \right)
\]
which proves the lemma due to~\eqref{eqn-decreasingdiff}.
\end{proof}

\end{comment}

\subsection{Characterizing the law of $(L,R)$ run backwards from $\tau$ to $\xi$}
\label{sec-reverse}

Recall that $\xi$ is the time at which $R$ attains its minimum on $[0,\tau)$, equivalently the time of the last running minimum of $R$ before time $\tau$.
The result of Proposition \ref{stochasticdomination} reduces the task of proving of Proposition \ref{onebubble} from showing that $\BB{E} \log(L_{\tau} - L_{\sigma}) > 0$ to showing that $\BB{E} \log(R_{\tau^-} - R_{\sigma}) > 0$.  The latter is a more tractable quantity since the definition $\sigma$ is, in some sense, more closely tied to the process $R$. To analyze this random variable, we first apply the following proposition, which follows immediately from known results in the L\'evy process literature.

\begin{prop}
The regular conditional joint distribution of the processes $\{L_{\tau^-} - L_{(\tau - t)^-} \}_{t \in [0,\tau - \xi]}$ and $\{R_{\tau^-} - R_{(\tau - t)^-} \}_{t \in [0,\tau - \xi]}$ given $\{R_{\tau^-} - R_{\xi} = r\}$ is equal to the law of $(L,R)$ stopped at the first time the process $R$ hits level $r$.
\label{reversedLevyprocesseslaw}
\end{prop}

\begin{proof}
\cite[Theorem 2]{bertoin-savov-duality} identifies the regular conditional distribution of $\{1+ R_{(\tau - t)^-} \}_{t \in [0,\tau)}$ given $\{1+R_{\tau^-} =x\}$ as that of a $\kappa/4$-stable Levy process with only positive jumps started at $x$ and conditioned to stay positive, run until the last exit time of this process from $[0,1]$. By \cite[Theorem 5]{chaumont-decomp} (along with the remark just before Proposition 2 in that paper), the regular conditional distribution of the latter process run until the (a.s.\ unique) time at which it attains its minimal value, conditioned on its minimal value being equal to $y < x$, is that of a $\kappa/4$-stable Levy process with only positive jumps started at $x$ and run until the first time when it hits $y$. 
Hence, the regular conditional law of $\{1+ R_{(\tau - t)^-} \}_{t \in [0,\tau - \xi]}$ given $\{1+R_{\tau^-} =x\} \cap \{1 + R_\xi = y\}$ is the same as the law of $x - R$ run until the first time when it hits $y$. 
This implies that the regular conditional law of $\{    R_{\tau^-} - R_{(\tau - t)^-} \}_{t \in [0,\tau)}$ given $\{1+R_{\tau^-} =x\} \cap \{1 + R_\xi = x  - r\}$  is the same as the law of $  R$ run until the first time when it hits $r$. 
Averaging over the possible values of $x$ and using that $L$ is independent from $R$ and our conditioning depends only on $R$ now gives the statement of the lemma. 
\end{proof}

Proposition~\ref{reversedLevyprocesseslaw} immediately implies the following corollary.

\begin{cor}
The regular conditional distribution of $R_{\tau^-} - R_{\sigma}$ given $\{R_{\tau^-} - R_{\xi} = r\}$ is equal to the law of the value of $R$ at the time $\theta_r$ of the last simultaneous running supremum of $(L,R)$ before $R$ hits the level $r$. In particular, since $R_{\theta_r} \eqD r R_{\theta_1}$ by scaling,
\begin{equation}
\BB{E}  \log(R_{\tau^-} - R_{\sigma}) = \BB{E}  \log(R_{\tau^-} - R_{\xi}) + \BB{E} \log(R_{\theta_1}).
\label{twoterms}
\end{equation}
\label{reversedLevyprocessescorollary}
\end{cor}

\subsection{Computing the expectations of $\log(R_{\tau^-} - R_{\xi})$ and $\log(R_{\theta_1})$}
\label{sec-compute}

To finish the proof of Theorem \ref{onebubble}, we compute the right-hand side of~\eqref{twoterms} and show it is non-negative for $\kappa \in (4,\kappa_0]$. We treat the two terms separately.

\begin{lem} \label{lem-term1}
One has $\BB{E}  \log(R_{\tau^-} - R_{\xi})  = \pi \cot(\pi \kappa/4)$.
\end{lem} 
\begin{proof}
The law of $\log(R_{\tau^-} - R_{\xi})$ is given explicitly in the literature:~\cite[Example 7]{dk-overshoot} gives the explicit joint density\footnote{Note that we are applying the formula in~\cite{dk-overshoot} to the process $-R$, and setting $x = 1$.  The positivity parameter $\rho$ associated to $-R$ that appears in the formula in~\cite{dk-overshoot} is defined as $\BB{P}(-R_1 \geq 0)$. Since $-R$ is a $\kappa/4$-stable process with only positive jumps, $\rho = 1- 4/\kappa$ (page 218 of~\cite{bertoin-book}).  As a result, the power of the $v-y$ term in the density equals zero, and so that term vanishes from the expression.}
\begin{align}
&\BB{P}\left(
-1 - R_{\tau} \in du,
R_{\tau^-} + 1 \in dv,
R_{\xi} + 1 \in dy \right) \notag \\
&\qquad = \frac{\kappa}{4} \left( 1- \frac{\kappa}{4} \right) \frac{\sin{(\pi \kappa/4)}}{\pi}  \frac{(1 - y)^{\kappa/4 -2}}{(v + u)^{\kappa/4 + 1}} \, du \,dv \, dy
\label{dk}
\end{align}
for $u>0$, $y \in [0,1]$, and $v \geq y$.
Substituting $v =y+w$ and integrating out $u$ gives
\[
\BB{P}\left(
R_{\tau^-} - R_{\xi} \in dw,
R_{\xi} + 1 \in dy \right) = 
\left( 1- \frac{\kappa}{4} \right) \frac{\sin{(\pi \kappa/4)}}{\pi}  \frac{(1 - y)^{\kappa/4 -2}}{(y+w)^{\kappa/4}} \, dw \, dy .
\]
This last density has antiderivative $\frac{\sin{(\pi \kappa/4)}}{\pi} \frac{(1 - y)^{\kappa/4-1} (w + y)^{1 - \kappa/4}}{1 + w}$ with respect to the $y$ variable, so 
\eqb
\BB{P}\left(
R_{\tau^-} - R_{\xi} \in dw \right) = -
\frac{\sin{(\pi \kappa/4)}}{\pi} \frac{w^{1-\kappa/4}}{1+w} \, dw .
\eqe
Therefore, using the well-known identities for the Beta function $B(p,q)$ (see, e.g., Section 15.02 of~\cite{jeffreys-beta-function})
\begin{equation}
B(p,q) = \frac{\Gamma(p) \Gamma(q)}{\Gamma(p+q)} = \int_0^1 x^{p-1} (1-x)^{q-1} \,dx \qquad p,q>0 \label{first-Beta-identity}
\end{equation}
and
\begin{equation}
B(p,1-p) = \frac{\pi}{\sin(p\pi)} \qquad 0<p<1 \label{second-Beta-identity}
\end{equation}
we get
\begin{align}
\BB{E} \log(R_{\tau^-} - R_{\xi})
\nonumber &=-\frac{\sin{(\pi \kappa/4)}}{\pi} \int_0^{\infty} \log(w) \frac{w^{1-\kappa/4}}{1+w} \, dw \\
\nonumber &=  \frac{\sin{(\pi \kappa/4)}}{\pi} \left. \frac{\partial}{\partial \beta} \left( \int_0^{\infty} \frac{w^{1-\beta}}{1+w} \, dw \right) \right|_{\beta=\kappa/4}\\
\nonumber &=  \frac{\sin{(\pi \kappa/4)}}{\pi} \left. \frac{\partial}{\partial \beta} \left( \int_0^1 (1-v)^{1-\beta} v^{\beta-2}  \, dv \right) \right|_{\beta=\kappa/4} \qquad \text{$v = (1+w)^{-1}$}\\
\nonumber &=  \frac{\sin{(\pi \kappa/4)}}{\pi} \left. \frac{\partial}{\partial \beta} \left(  B(2-\beta,\beta-1)  \right) \right|_{\beta=\kappa/4}\qquad \text{by \eqref{first-Beta-identity}}\\
\nonumber &= - \sin{(\pi \kappa/4)} \left. \frac{\partial}{\partial \beta} \left( \frac{1}{\sin(\pi \beta)} \right) \right|_{\beta=\kappa/4} \qquad \text{by \eqref{second-Beta-identity} }\\
&= \pi \cot(\pi \kappa/4).  \label{firstterm}
\end{align}
\end{proof}

We now turn to analyzing the second term in~\eqref{twoterms}. 
 
\begin{lem} \label{lem-term2}
One has $\BB{E} \log R_{\theta_1}  = \psi(2-\kappa/4) - \psi(1)$, where $\psi(x) = \frac{\Gamma'(x)}{\Gamma(x)}$ denotes the digamma function (as in Theorem~\ref{main}).
\end{lem}

We will first compute the law of $R_{\theta_1}$.

\begin{lem} \label{lem-arcsine}
The law of $R_{ \theta_1}$ is given by the generalized arcsine distribution,
\eqb \label{eqn-arcsine}
\BB{P}(R_{\theta_1} \in dx ) = \frac{\sin \pi\left( 2 - \kappa/4 \right)}{\pi} x^{1 - \kappa/4} (1-x)^{\kappa/4-2} \, dx.
\eqe 
\end{lem} 
\begin{proof}
We will deduce the lemma from the arsine law for a certain stable subordinator.
Recall that $\theta_1$ is defined as the time of the last simultaneous running supremum of $(L,R)$ before $R$ hits the level $r$.  The simultaneous running suprema of $(L,R)$ are easier to analyze by expressing the law as $(L,R)$ in terms of a pair of correlated Brownian motions with a particular subordination.  

Suppose that $(\wt{L},\wt{R})$ is a planar Brownian motion with $\text{var}(\wt{L}_1)=\text{var}(\wt{R}_1)= \frac{1}{2} - \frac{p}{2}$ and $\text{cov}(\wt{L}_1,\wt{R}_1) = \frac{p}{2}$, where $p = -\cos(4\pi/\kappa)/(1 - \cos(4\pi/\kappa))$. For times $0 < s < t$, if $\wt{L}_r > \wt{L}_s$ and $\wt{R}_r > \wt{R}_s$ for all $r \in (s,t]$, then we say that $s$ is an \emph{ancestor} of $t$.  A time $t$ that does not have an ancestor is called \emph{ancestor free}. 
The set of ancestor free times is an uncountable set and has zero Lebesgue measure by~\cite[Lemma 1]{shimura-cone}.

Using standard Brownian motion techniques, it is shown in~\cite[Proposition 10.3]{wedges} that we can define a nondecreasing c\`adl\`ag process $\ell_t$ which is adapted to the filtration of $(\wt{L}_t,\wt{R}_t)$ and which measures the local time for $(\wt{L}_t,\wt{R}_t)$ at the ancestor-free times. Moreover, if $T_u = \inf\left\{ t \geq 0 : \ell_t > u\right\}$ is the right-continuous inverse of $\ell_t$, then the range of $u\mapsto T_u$ is the set of ancestor free times and 
the pair $( \wt{L}_{T_u},\wt{R}_{T_u})$ has the same joint law as the pair of $\kappa/4$-stable processes $- (L,R)$ (which have only upward jumps), modulo a deterministic scaling factor (see Remark~\ref{remark-scaling}).  

In particular, the random variable $R_{\theta_1}$ has the same law as $-\wt{R}_{\wt{\theta}_1}$, where $\wt{\theta}_1$ is the time of the last simultaneous running infimum of the correlated planar Brownian motion $(\wt{L},\wt{R})$ before $\wt{R}$ hits the level $-1$. 

The set of values of $-\wt{R}$ at the simultaneous running infima of $(\wt{L},\wt{R})$ is clearly regenerative; by scale invariance, it has the law of a stable subordinator. 
We claim that the index of this subordinator is $2-\kappa/4$. Once this is established, the arcsine law for subordinators~\cite[Proposition 3.1]{bertoin-sub} shows that the law of $-\wt R_{\wt\theta_1} \overset{\mcl L}{=} R_{\theta_1}$ is given by the right side of~\eqref{eqn-arcsine}, which concludes the proof. 
 
To determine the index of the above subordinator, it is enough to compute the a.s.\ Hausdorff dimension of its range.  First, we recall the following definition.

\begin{defn} \label{def-conetime}
A \textit{$\pi/2$-cone time} of an $\BB{R}^2$-valued process $(X,Y)$ is a time $t$ for which, for some choice of $\epsilon > 0$, we have $X_s > X_t$ and $Y_s > Y_t$ for all $s \in (t - \epsilon , t)$.  The largest such interval $(t - \epsilon , t)$ is called a \textit{$\pi/2$-cone interval} of $(X,Y)$.
\end{defn}

The set $\mcl R$ of times of the simultaneous running infima of  $(\wt{L},\wt{R})$ is precisely the set of $\pi/2$-cone times of $( \wt{L}, \wt{R})$
with the property that 0 is contained in the corresponding cone interval.  Thus,~\cite[Theorem 1]{evans-cone} (applied to a linear transformation of $(\wt L , \wt R)$ chosen so that the coordinates are independent) implies that the Hausdorff dimension of $\mcl R$ is $1 - \kappa/8$ almost surely. On the other hand, $\wt R(\mcl R ) = S^{-1}(\mcl R)$, where for $r \geq 0$, $S_r := \inf\{t > 0 : \wt R_t = -r\}$. Since $S$ is a $1/2$-stable subordinator,~\cite[Theorem 4.1]{hawkes-uniform} implies that $\op{dim}(R(\mcl R)) = 2\dim \mcl R = 2-\kappa/4$. Hence the set of values of $-\wt{R}$ at the simultaneous running infima of $(\wt{L},\wt{R})$ is an index $2 - \kappa/4$ subordinator.  
\end{proof}

\begin{proof}[Proof of Lemma~\ref{lem-term2}]
Using Lemma~\ref{lem-arcsine}, we compute
\begin{align}
\BB{E} \log R_{\theta_1}  
\nonumber &= \frac{1}{B(2 - \kappa/4, \kappa/4 - 1)} \int_0^1 \log{x} \cdot  x^{1 - \kappa/4} (1-x)^{\kappa/4-2} \, dx \qquad \text{by \eqref{second-Beta-identity} }\\
\nonumber &= \frac{1}{B(2 - \kappa/4, \kappa/4 - 1)} \int_0^1  \left. \frac{\partial}{\partial \beta}\left( x^{\beta - 1} (1-x)^{\kappa/4-2} \right) \right|_{\beta = 2 - \kappa/4} \, dx \\
\nonumber &= \frac{1}{B(2 - \kappa/4, \kappa/4 - 1)}   \left. \frac{\partial B(\beta , \kappa/4 - 1)}{\partial \beta} \right|_{\beta = 2 - \kappa/4} \qquad \text{by \eqref{first-Beta-identity} }\\
\nonumber &=  \left. \frac{\partial \log B(\beta , \kappa/4 - 1)}{\partial \beta} \right|_{\beta = 2 - \kappa/4} \\
\nonumber &=  \left. \frac{\partial \log \Gamma (\beta)}{\partial \beta} \right|_{\beta = 2 - \kappa/4} 
-   \left. \frac{\partial \log \Gamma (\beta + \kappa/4 - 1)}{\partial \beta} \right|_{\beta = 2 - \kappa/4} \quad \text{by \eqref{first-Beta-identity}} \\
&= \psi(2-\kappa/4) - \psi(1).   \label{secondterm}
\end{align}
\end{proof}

\begin{proof}[Proof of Theorem~\ref{onebubble}]
Plugging Lemmas~\ref{lem-term1} and~\ref{lem-term2} into~\eqref{twoterms} gives
\[
\BB{E}  \log(R_{\tau^-} - R_{\sigma}) = \BB{E}  \log(R_{\tau^-} - R_{\xi}) + \BB{E} \log(R_{\theta_1}) = \pi \cot(\pi \kappa/4) + \psi(2-\kappa/4) - \psi(1).
\]
The latter is a monotonically decreasing function of $\kappa$, and equals zero for $\kappa \approx 5.6158$.  Combining this with Proposition~\ref{stochasticdomination} proves Theorem \ref{onebubble}. 
\end{proof}

\section{Proof of Theorem \ref{onebubbleconverse}}
\label{sec-converseproof}

To prove Theorem \ref{onebubbleconverse}, we first characterize the limiting law of $L$ in the Skorohod topology as $\kappa$ tends to $8$.\footnote{The random variables considered in this section (such as $L$, $R$, and $\tau$) are all defined for each $\kappa$; however, to avoid clutter, we will not indicate this dependence on $\kappa$ in our notation.} To do this, we first need to specify the exact law of $L$. Recall from Remark~\ref{remark-scaling} that we have thus far only specified the law of $L$ up to a multiplicative constant.  Since changing this constant does not change the law of  the random variable $\log\left( \sup_{t \in \mathcal{M}} (L_t - L_{\sigma(t)} )\right)$, we may assume without loss of generality that 
$L$ is chosen to have characteristic function
\begin{equation}
\BB{E} e^{i \lambda L_t} = e^{t (i \lambda)^{\kappa/4}} = \exp\left(t |\lambda|^{\kappa/4} \left[ \cos \frac{\pi \kappa}{8}  + i \text{sgn}(\lambda) \sin \frac{\pi \kappa}{8} \right] \right),
\label{charfunction}
\end{equation}
so that
\begin{equation}
\BB{E} e^{\lambda L_t} = e^{t \lambda^{\kappa/4}}
\label{laplacetransform}
\end{equation}
for $\lambda \geq 0$~\cite{bernyk}.
For this choice of $L$, we have the following convergence result:

\begin{prop}
The process $L$ defined by~\eqref{charfunction} converges to $\sqrt{2} B$ in the Skorohod topology, where $B$ is a standard Brownian motion.
\label{convtoBM}
\end{prop}
\begin{proof}
By the expression~\eqref{charfunction} for the characteristic function of $L_t$, one has $L_t \rta \sqrt 2 B_t$ in law for each fixed $t\geq 0$.
The proposition therefore follows from a standard convergence criterion for L\'evy processes; see, e.g.,~\cite[Theorem 13.17 or Exercise 14.3]{kallenberg}. 
\end{proof}

Proposition~\ref{convtoBM} allows us to show that $\sup_{t \in \mathcal{M}} (L_t - L_{\sigma(t)})$ converges to zero in distribution as $\kappa \rightarrow 8$, since the intervals $[\sigma(t),t]$ are all degenerate in the $\kappa \rightarrow 8$ limit by well-known properties of Brownian motion.  Formally, we have the following corollary:

\begin{cor}
The random variable
\[
\max_{t \in \mathcal{M}} \; (t - \sigma(t))
\]
converges to zero in law as $\kappa \rightarrow \infty$.
\label{noconetimes}
\end{cor}
\begin{proof}
By Proposition \ref{convtoBM}, the law of $(L,R)$ converges as $\kappa \rightarrow 8$ to $(\sqrt{2} B_1, \sqrt{2} B_2)$, where $B_1$ and $B_2$ are independent standard Brownian motions. By Skorohod's representation theorem, we can represent the distributions of $(L,R)$ for $\kappa \in (4,8)$ on the same probability space so that this convergence occurs almost surely. Since a linear Brownian motion a.s.\ enters $(-\infty,-1)$ immediately after hitting $-1$, we see that $\tau$ converges to a limit almost surely as $\kappa \rightarrow 8$.  Thus, if we assume for contradiction that $\max_{t \in \mathcal{M}} (t - \sigma(t))$ does not tend to zero as $\kappa \rightarrow 8$, we can choose a subsequence $\kappa_n$ tending to $8$ and, for each $n$, an element $t_n$ in the set $\mathcal{M}$ corresponding to $\kappa = \kappa_n$, such that the intervals $[\sigma(t_n),t_n]$ converge to an interval $[a,b]$ with $a<b$ as $n \rightarrow \infty$.  By the almost sure convergence of the processes $L$ in the Skorohod topology, the continuity of the limiting process $(\sqrt{2} B_1, \sqrt{2} B_2)$, and the definition of $\sigma(t_n)$ (Notation~\ref{sigmanot}) the interval $[a,b]$ is a $\frac{\pi}{2}$-cone interval for $(\sqrt{2} B_1, \sqrt{2} B_2)$ (Definition~\ref{def-conetime}), which is a contradiction since an uncorrelated planar Brownian motion a.s.\ does not have any $\frac{\pi}{2}$-cone times~\cite[Theorem~1]{shimura-cone}.
\end{proof}

Proposition~\ref{convtoBM} together with Corollary \ref{noconetimes} implies that $\sup_{t \in \mathcal{M}} (L_t - L_{\sigma(t)})$ converges to zero in distribution as $\kappa \rightarrow 8$.  Hence, for each fixed $K>0$,
\[
 \log\left( \sup_{t \in \mathcal{M}} (L_t - L_{\sigma(t)} )\right) \vee (-K) \rightarrow -K
\]
in distribution as $\kappa \rightarrow 8$. So, to prove that the expectation of $\log\left( \sup_{t \in \mathcal{M}} (L_t - L_{\sigma(t)}) \right) $ is negative for $\kappa$ sufficiently close to $8$, it suffices to check the following uniform integrability result:

\begin{lem}
For each fixed $K>0$ and $\kappa' \in (4,8)$, the set of random variables $\max_{s \in [0,\tau]} \log|L_s| \vee (-K) $ for $\kappa \in [\kappa',8)$ is uniformly integrable.
\end{lem}

\begin{proof}
To prove uniform integrability, it suffices to show that the expectation of 
\[\varphi\left( \left| \max_{s \in [0,\tau]} \log|L_s| \vee (-K) \right| \right)\]
is bounded uniformly in $\kappa \in [\kappa',8)$, where $\varphi(x) = e^{q x}$ for some $q>0$. Proving this, in turn, reduces to showing that the expectation of \[\max_{s \in [0,\tau]} |L_s|^q \] is bounded uniformly in $\kappa \in [\kappa',8)$ for some $q>0$.  We will prove such a bound using moment bounds on $L_1$ and $\tau$. 

First, simplifying equation (8.26) on page 292 of~\cite{paolella} for $\alpha = \kappa/4$, $\beta = -1$ and $X = -\cos(\pi \kappa/4) L_1$ yields\footnote{The random variable $X$ has characteristic function given by equation (8.8) on page 281 of~\cite{paolella} with $c=1$; comparing this characteristic function with that of $L_1$ yields the correct scaling $X = -\cos(\pi \kappa/4) L_1$.}
\[
\BB{E}( |L_1|^r ) = \frac{\Gamma\left( 1 - \frac{4r}{\kappa} \right)}{\Gamma(1-r)} \left( - \cos\left( \frac{\pi \kappa}{8}\right) \right)^{-r + 4r/\kappa}
\]
The latter is bounded uniformly in $\kappa \in [\kappa',8)$ for each fixed $r < \kappa'/4$.  
As for $\tau$,~\cite{peskir} derives the following series representation for the density $f_{\tau}$ of $\tau$:
\begin{align*}
f_{\tau}(t) &= \frac{1}{\pi t^{2 - 4/\kappa}} \sum_{n=1}^{\infty} \left[ (-1)^{n-1} \sin(4 \pi/\kappa) \frac{\Gamma(n-4/\kappa)}{\Gamma( n \kappa/4 - 1)} \frac{1}{t^{n-1}}  \right.\\ & \qquad \qquad \qquad\left. - \sin\left( \frac{4 n\pi}{\kappa} \right) \frac{\Gamma(1 + 4n/\kappa)}{n!} \frac{1}{t^{4(n+1)/\kappa - 1}} \right], \qquad \forall t>0 .
\end{align*}
Therefore, for $t \geq 1$ and $\kappa \in [\kappa',8)$,
\begin{align*}
|f_{\tau}(t)| 
&\leq \frac{1}{\pi t^{2-4/\kappa}} \sum_{n=1}^{\infty} \left[ \frac{\Gamma(n-4/\kappa)}{\Gamma(n \kappa/4 - 1)} + \frac{\Gamma(1+4n/\kappa)}{n!} \right]\\
&\leq \frac{1}{\pi t^{2-4/\kappa}} \sum_{n=1}^{\infty} \left[ \frac{(n-1)!}{\left\lfloor n \kappa'/4 - 2 \right\rfloor !} + \frac{\left\lfloor 4n/\kappa'\right\rfloor!)}{n!} \right]  
 \leq \frac{C_{\kappa'}}{t^{2-4/\kappa}} 
\end{align*}
Hence, for any choice of $0< q < \kappa'/4-1$, the quantity $\BB{E}(\tau^{4q/\kappa})$ is bounded uniformly in $\kappa \in [\kappa',8)$.
Thus, fixing $0<q < \kappa'/4-1$ and $1< r < \kappa'/4$, we have
\begin{align*}
\BB{E}\left( \max_{s \in [0,\tau]} |L_s|^q \right)
&=
\BB{E}(\tau^{4q/\kappa}) \BB{E}\left( \max_{s \in [0,1]} |L_s|^q \right) \qquad \text{by scaling (since $\tau,L$ are independent)}
\\&=
\BB{E}(\tau^{4q/\kappa}) \BB{E}\left( \max_{s \in [0,1]} |L_s|^r \right)^{q/r}
\\&=
\BB{E}(\tau^{4q/\kappa}) \left( \frac{r}{1-r} \right)^q \left(\BB{E} \left( |L_1|^r \right) \right)^{q/r}, \qquad \text{by Doob's inequality}
\end{align*}
which is bounded uniformly in $\kappa \in [\kappa',8)$. This completes the proof.
\end{proof}

\section{Open problems}
\label{sec-open}

Consider the following three properties the adjacency graph of bubbles of the SLE$_\kappa$ curves $\eta$:
\begin{enumerate}[label=\textbf{(\Roman*)}]
\item
The graph is a.s.\ connected, i.e., there a.s.\ exists a finite path joining any pair of bubbles.
\label{1}
\item
Almost surely, there exists a path of bubbles from any fixed bubble to $\infty$ which are formed at increasing times (i.e., the path hits the bubbles in the order in which they are formed by the curve and only finitely many bubbles in the path intersect any given compact subset of $\ol{\BB H}$). 
\label{2}
\item
There exists an $(L,R)$-Markovian path started at any stopping time $\zeta$ for $(L,R)$ at which $\eta$ forms a bubble (Definition~\ref{def-markovpath}).
\label{3}
\end{enumerate}
Property \ref{3} is clearly stronger than \ref{2}; the proof of Lemma~\ref{suffcondition} in fact shows that \ref{2} is stronger than \ref{1}.  In Theorem~\ref{path}, we showed that \ref{3} (and hence also \ref{2} and \ref{1}) hold for $\kappa \in (4,\kappa_0]$, and in Theorem~\ref{largekappa} we showed that \ref{3} fails for $\kappa$ sufficiently close to $8$.  

It is of interest to determine the exact set of values of $\kappa \in (4,8)$ for which each of the above three properties hold. 
As mentioned in the introduction, our intuition suggests that it is easier for the adjacency graph to be connected when $\kappa$ is closer to $4$. 
This means that for each of the above three properties, there should exist a critical $\kappa^* \in [\kappa_0 ,8]$ for which the property  holds for $\kappa \in (4,\kappa^*)$ but fails for $\kappa \in (\kappa^*,8)$.

For property~\ref{3}, one might guess that $\kappa^* = 6$, since this is the only ``special'' value of $\kappa$ in the range $(\kappa_0,8)$ and our proof of Theorem~\ref{path}, which gives $\kappa_0 \approx 5.6158$, seems to be reasonably close to optimal. But, we would not be surprised if this does not turn out to be true.
It would be somewhat odd if there exists values of $\kappa$ for which \ref{2} holds but \ref{3} fails, since this would mean that there exist paths to infinity in the adjacency graph but that such paths cannot be found in a Markovian way. Hence $\kappa^* = 6$ might also be a reasonable guess for the critical value for property~\ref{2}.   
For condition~\ref{1}, we are not sure if $\kappa^* = 8$ (i.e., the graph is connected for all $\kappa$) or if $\kappa^* <8$; we would not be surprised either way. 
Our results indicate that it might be difficult to prove connectedness for $\kappa$ close to 8 (if this is indeed true) since one would have to find a way of producing paths which is not Markovian with respect to $(L,R)$.

\bibliography{cibib,slecomponentssupplementarybib}
\bibliographystyle{hmralphaabbrv}
 
\end{document}